\documentclass[11pt,reqno]{amsart}
\usepackage{enumerate}
\usepackage{amsrefs}
\usepackage{amsfonts, amsmath, amssymb, amscd, amsthm, bm, cancel,mathrsfs}
\usepackage{url}
\usepackage{graphicx}
\usepackage[
linktocpage=true,colorlinks,citecolor=magenta,linkcolor=blue,urlcolor=magenta]{hyperref}
\usepackage{multicol}
\usepackage{comment}
\usepackage{tikz}
\usepackage[margin=1in]{geometry}

\newtheorem{thm}{Theorem}[section]

\newtheorem*{thmA'}{Theorem A'}

\newtheorem{cor}[thm]{Corollary}
\newtheorem{con}[thm]{Conjecture}
\newtheorem{lem}[thm]{Lemma}

\newtheorem{prop}[thm]{Proposition}

\theoremstyle{definition}

\theoremstyle{remark}
\newtheorem{rem}{Remark}

\numberwithin{equation}{section}
\allowdisplaybreaks

\renewcommand{\(}{\left(}
\renewcommand{\)}{\right)}
\renewcommand{\widetilde}{\tilde}
\renewcommand{\-}{\overline}

\renewcommand{\a}{\alpha}
\renewcommand{\b}{\beta}
\newcommand{\g}{\gamma}
\renewcommand{\d}{\delta}
\newcommand{\e}{\varepsilon}
\renewcommand{\k}{\kappa}
\renewcommand{\l}{\lambda}

\newcommand{\D}{\Delta}

\renewcommand{\t}{\theta}
\newcommand{\s}{\sigma}

\newcommand{\G}{\Gamma}

\newcommand{\ra}{\rightarrow}

\begin{document}
	\title[Alexandrov-Fenchel inequalities in the half-space]{A complete family of Alexandrov-Fenchel inequalities for convex capillary hypersurfaces in the half-space}	
\author[Y. Hu]{Yingxiang Hu}
\address{School of Mathematical Sciences, Beihang University, Beijing 100191, P.R. China}
\email{\href{mailto:huyingxiang@buaa.edu.cn}{huyingxiang@buaa.edu.cn}}
\author[Y. Wei]{Yong Wei}
\address{School of Mathematical Sciences, University of Science and Technology of China, Hefei 230026, P.R. China}
\email{\href{mailto:yongwei@ustc.edu.cn}{yongwei@ustc.edu.cn}}
\author[B. Yang]{Bo Yang}
\address{Department of Mathematical Sciences, Tsinghua University, Beijing 100084, P.R. China}
\email{\href{mailto:ybo@tsinghua.edu.cn}{ybo@tsinghua.edu.cn}}
\author[T. Zhou]{Tailong Zhou}
\address{School of Mathematics, Sichuan University, Chengdu 610065, Sichuan, P.R. China}
\email{\href{mailto:zhoutailong@scu.edu.cn}{zhoutailong@scu.edu.cn}}
	
\thanks{The research was surpported by National Key Research and Development Program of China 2021YFA1001800 and 2020YFA0713100, National Natural Science Foundation of China NSFC11721101 and NSFC12101027.
}
	\begin{abstract}
In this paper, we study the locally constrained inverse curvature flow for hypersurfaces  in the half-space with $\theta$-capillary boundary, which was recently introduced by Wang-Weng-Xia \cite{Wang-Weng-Xia2022}. Assume that the initial hypersurface is strictly convex with the contact angle $\theta\in (0,\pi/2]$. We prove that the solution of the flow remains to be strictly convex for $t>0$, exists for all positive time and converges smoothly to a spherical cap. As an application, we prove a complete family of Alexandrov-Fenchel inequalities for convex capillary hypersurfaces in the half-space with the contact angle $\theta\in(0,\pi/2]$.
Along the proof, we develop a new tensor maximum principle for parabolic equations on compact manifold with proper Neumann boundary condition.
	\end{abstract}	
	\maketitle
		\tableofcontents

\section{Introduction}
The classical Alexandrov-Fenchel inequalities for convex domains in Euclidean space are important in convex geometry. For a smooth bounded domain $\widehat{\Sigma}\subset \mathbb R^{n+1}$, we denote $\Sigma=\partial \widehat{\Sigma}$ the boundary of $\widehat{\Sigma}$. Let $\k(x)=(\k_1(x),\cdots,\k_n(x))$ be the principal curvatures of $\Sigma\subset \mathbb R^{n+1}$ at the point $x$, and $H_k(\k)$ is the $k$th normalized mean curvature of $\Sigma$. The quermassintegrals of $\widehat{\Sigma}$ are given by
\begin{align}\label{s1:def-quermassintegral-closed}
W_0(\widehat{\Sigma}):=|\widehat{\Sigma}|, \quad W_{k+1}(\widehat{\Sigma}):=\frac{1}{n+1}\int_{\Sigma} H_{k} dA, \quad 0\leq k\leq n.
\end{align}
The Alexandrov-Fenchel inequalities for quermassintegrals (see \cite{Schn}) of a convex domain $\widehat{\Sigma}$ state as
\begin{align}\label{s1:AF-inequality}
\frac{W_k(\widehat{\Sigma})}{|\mathbb B^{n+1}|} \geq \(\frac{W_\ell(\widehat{\Sigma})}{|\mathbb B^{n+1}|}\)^\frac{n+1-k}{n+1-\ell}, \quad 0\leq \ell<k\leq n,
\end{align}
with equality holds if and only if $\widehat{\Sigma}$ is a ball. The inequality \eqref{s1:AF-inequality} is a natural generalization of the classical isoperimetric inequality, which corresponds to the case $k=1, \ell=0$ of \eqref{s1:AF-inequality}. Using a quermassintegral preserving curvature flow, McCoy \cite{McC05} provides a new proof of the inequality \eqref{s1:AF-inequality} for smooth convex domains. Later,  Guan and Li \cite{GL09} applied the inverse curvature flow \cite{Ger1990,Urbas1991} to show that the inequality \eqref{s1:AF-inequality} holds for smooth $(k-1)$-convex and starshaped domains. Using the similar idea, there are many recent work on the Alexandrov-Fenchel type inequalities for smooth domains in the hyperbolic space \cite{ACW20,AHL2020,BGL,GL-2021,GWW14,HL18,HLW2020,LWX14,Wang-Xia2014} and in the sphere \cite{CS22,ChenGLS22,MS15,WeiX15}.

In this paper, we focus on the generalizations of \eqref{s1:AF-inequality} to the hypersurfaces with capillary boundary in the half-space $\overline{\mathbb{R}}_+^{n+1}$. Let $\Sigma\subset \-{\mathbb R}^{n+1}_{+}$ be a smooth properly embedded hypersurface in the half-space $\-{\mathbb R}^{n+1}_{+}$ with boundary $\partial \Sigma$ supported on $\partial \-{\mathbb R}^{n+1}_{+}$, where by proper embeddedness we mean that $\operatorname{int}(\Sigma)\subset \mathbb R^{n+1}_{+}$ and $\partial \Sigma \subset \partial \-{\mathbb R}^{n+1}_{+}$. Let $\widehat{\Sigma}$ be the bounded domain enclosed by $\Sigma$ and the hyperplane $\partial\-{\mathbb R}^{n+1}_{+}$. Then the boundary of $\widehat{\Sigma}$ consists of two parts: one is $\Sigma$ and the other one lying on $\partial \overline{\mathbb{R}}^{n+1}_+$ is denoted by $\widehat{\partial\Sigma}$. See Figure \ref{fig1} in \S \ref{sec:2}. We also denote by $\partial\Sigma=\Sigma\cap \widehat{\partial\Sigma}$ as the common boundary of $\Sigma$ and $\widehat{\partial\Sigma}$. Suppose that $\Sigma$ intersects with $\partial \overline{\mathbb{R}}_+^{n+1}$ at a constant contact angle $\theta\in (0,\pi)$. Then the following relative isoperimetric inequality in the half-space holds (see \cite[\S 19.3]{Maggi}):
\begin{align}\label{s1.iso}
\frac{|\Sigma|-\cos\theta |\widehat{\partial\Sigma}|}{|\mathbb S^n_{\theta}|-\cos\theta|\widehat{\partial \mathbb{S}^n_\theta}|} \geq \(\frac{|\widehat{\Sigma}|}{|\mathbb B_{\theta}^{n+1}|}\)^\frac{n}{n+1}.
\end{align}
The equality holds in \eqref{s1.iso} if and only if $\Sigma$ is homothetic to a spherical cap $\mathbb S^n_{\theta}$ with contact angle $\theta$. See \eqref{s2:spherical-caps} for the definition of the spherical caps, where $\mathbb S^n_{\theta}=C_{1,\theta}(e)$.

Recently, the following geometric functionals were introduced by Wang, Weng and Xia \cite{Wang-Weng-Xia2022}, which can be considered as the suitable quermassintegrals for capillary hypersurfaces in the half-space $\-{\mathbb{R}}^{n+1}_{+}$ with constant contact angle $\t$:
\begin{align}\label{s1:quermassintegral-capillary}
\mathcal{V}_{0,\t}(\widehat{\Sigma}):=&|\widehat{\Sigma}|, \nonumber\\
\mathcal{V}_{1,\t}(\widehat{\Sigma}):=&\frac{1}{n+1}(|\Sigma|-\cos \t |\widehat{\partial \Sigma}|), \nonumber\\
\mathcal{V}_{k+1,\t}(\widehat{\Sigma}):=&\frac{1}{n+1}\int_{\Sigma}H_k dA-\frac{\cos\t \sin^k\t}{n}\int_{\partial \Sigma}H_{k-1}^{\partial\Sigma}ds, \quad 1\leq k\leq n,
\end{align}
where $H_{k-1}^{\partial \Sigma}$ is the normalized $(k-1)$th mean curvature of $\partial \Sigma \subset \mathbb R^n$. In particular, one has
\begin{align*}
\mathcal{V}_{n+1,\t}(\widehat{\Sigma})=\frac{1}{n+1}\int_{\Sigma}H_ndA-\cos\t \sin^n\t \frac{\omega_{n-1}}{n(n+1)}=(n+1)\omega_\t,
\end{align*}
which can be considered as a Gauss-Bonnet type result for capillary hypersurfaces in $\-{\mathbb{R}}^{n+1}_{+}$  with contact angle $\t$, where $\omega_{n-1}=|\mathbb{S}^{n-1}|$ and $\omega_\theta=|\mathbb S^n_{\theta}|-\cos\theta|\widehat{\partial \mathbb{S}^n_\theta}|$.

The quermassintegrals defined above satisfy the following variational formulas (\cite[Theorem 1.1]{Wang-Weng-Xia2022}): Let $\Sigma_t\subset \-{\mathbb R}^{n+1}_{+}$ be a family of smooth, embedded, capillary hypersurfaces with a fixed contact angle $\t$, which are given by the embeddings $x(\cdot,t):M\ra \-{\mathbb R}^{n+1}_{+}$ and satisfy
\begin{align*}
\(\partial_tx\)^\bot=\mathcal{F}\nu
\end{align*}
for some speed function $\mathcal{F}$. Then for $0\leq k\leq n$,
\begin{align}\label{s1:variation-formula}
\frac{d}{dt}\mathcal{V}_{k,\t}(\widehat{\Sigma}_t)=\frac{n+1-k}{n+1}\int_{\Sigma_t}\mathcal{F} H_k dA_t,
\end{align}
and in particular
\begin{align*}
\frac{d}{dt}\mathcal{V}_{n+1,\t}(\widehat{\Sigma}_t)=0.
\end{align*}
Let $\mathbb B_\t^{n+1}:=\{x\in \mathbb B^{n+1}| \langle x,e_{n+1}\rangle >\cos\t \}$. The following conjecture was proposed in \cite[Conjecture 1.5]{Wang-Weng-Xia2022}, which could be viewed as the relative Alexandrov-Fenchel type inequalities for quermassintegrals in the half-space.
\begin{con}[\cite{Wang-Weng-Xia2022}]\label{s1.Ques}
For $n\geq 2$, let $\Sigma\subset \-{\mathbb R}^{n+1}_{+}$ be a convex hypersurface with capillary boundary supported on $\partial\-{\mathbb R}_{+}^{n+1}$ at a contact angle $\theta\in (0,\pi)$. Then there holds
	\begin{align}\label{s1:AF-inequality-capillary-conj}
	\frac{\mathcal{V}_{k,\t}(\widehat{\Sigma})}{|\mathbb B_\t^{n+1}|} \geq \(\frac{\mathcal{V}_{\ell,\t}(\widehat{\Sigma})}{|\mathbb B_\t^{n+1}|}\)^\frac{n+1-k}{n+1-\ell}, \quad 0\leq \ell<k\leq n
	\end{align}
with equality holding in \eqref{s1:AF-inequality-capillary-conj} if and only if $\Sigma$ is a spherical cap  in the half-space with $\t$-capillary boundary.
\end{con}

For this purpose, Wang, Weng and Xia \cite{Wang-Weng-Xia2022} introduced a new locally constrained inverse curvature flows: Let $e=-e_{n+1}$, where $e_{n+1}$ is the $(n+1)$th coordinate vector in $\mathbb R_+^{n+1}$. Let $x:M\ra \Sigma \subset \-{\mathbb R}^{n+1}_{+}$ be the embedding of $\Sigma=x(M)\subset \-{\mathbb R}^{n+1}_{+}$ with its boundary $x|_{\partial M}:\partial M \ra \partial \Sigma \subset \partial\overline{\mathbb R}^{n+1}_+$ and $\nu$ its unit normal vector field.  For $k=1,\cdots,n$, we consider a family of smooth immersions $x(\cdot,t)$ of capillary hypersurfaces in the half-space with contact angle $\t\in (0,\frac{\pi}{2}]$, which evolves by
\begin{align}\label{s1:BGL-flow}
\left\{\begin{aligned}(\partial_t x)^\bot=&\((1+\cos\t\langle \nu,e\rangle)\frac{H_{k-1}}{H_k}-\langle x,\nu\rangle\)\nu, &\text{in $M\times [0,T)$},\\
\langle \-N\circ x,\nu\rangle=&\cos(\pi-\t), &\text{on $\partial M\times [0,T)$},\\
x(\cdot,0)=&x_0(\cdot), &\text{on $M$}.
\end{aligned}\right.
\end{align}
The introduction of this flow is motivated by a similar flow proposed earlier by Brendle, Guan and Li \cite{BGL} for closed hypersurfaces in the space forms, and is based on the Minkowski type formula \cite[Proposition 2.5]{Wang-Weng-Xia2022}:
\begin{align}\label{s1:Minkowski-formula}
\int_{\Sigma}H_{k-1}(1+\cos\t\langle \nu,e\rangle) dA=\int_{\Sigma} H_k\langle x,\nu\rangle dA, \quad 1\leq k\leq n,
\end{align}
 for $\t$-capillary hypersurfaces in the half-space. It follows from the variation formula \eqref{s1:variation-formula} that
\begin{align}\label{s1:evol-V-ell}
\frac{d}{dt}\mathcal{V}_{k,\t}(\widehat{\Sigma}_t)=\frac{n+1-k}{n+1}\int_{\Sigma_t}\((1+\cos\t\langle \nu,e\rangle) H_{k-1}-\langle x,\nu\rangle H_{k}\)dA_t=0,
\end{align}
and for $1\leq \ell<k$
\begin{align}\label{s1:evol-V-k}
\frac{d}{dt}\mathcal{V}_{\ell,\t}(\widehat{\Sigma}_t)=&\frac{n+1-\ell}{n+1}\int_{\Sigma_t}\((1+\cos\t\langle \nu,e\rangle)\frac{H_{k-1}}{H_{k}}H_\ell-\langle x,\nu\rangle H_{\ell}\)dA_t \nonumber\\
  \geq &\frac{n+1-\ell}{n+1}\int_{\Sigma_t}\((1+\cos\t\langle e,\nu)\rangle H_{\ell-1}-\langle x,\nu\rangle H_{\ell}\)dA_t=0,
\end{align}
where we used the Newton-MacLaurin inequality \eqref{s2:NM-ineq}. That is, along the flow \eqref{s1:BGL-flow}, $\mathcal{V}_{k,\t}$ is preserved and $\mathcal{V}_{k,\t}$ is monotone increasing for $\ell<k$. Then the flow \eqref{s1:BGL-flow} could be used as a tool to establish the Alexandrov-Fenchel inequalities for convex capillary hypersurfaces in the half-space. The special case $k=n$ of the flow \eqref{s1:BGL-flow} has been considered in \cite{Wang-Weng-Xia2022}, and as a consequence, the Alexandrov-Fenchel inequality between $\mathcal{V}_{n,\t}$ and $\mathcal{V}_{\ell,\t}$ for $0\leq \ell<n$ is obtained.

The aim of this paper is to provide a solution to the Conjecture \ref{s1.Ques} for convex capillary hypersurfaces in the half-space with the contact angle $\theta\in(0,\pi/2]$. We shall prove the smooth convergence of the flow \eqref{s1:BGL-flow} for all $1\leq k \leq n$. As the key step, we prove the preservation of the convexity along the flow \eqref{s1:BGL-flow}. For this purpose, we develop a generalized tensor maximum principle on compact manifolds with proper Neumann boundary condition.
\begin{thm}\label{s1:thm-max principle}
	Let $\Sigma$ be a smooth compact manifold with boundary $\partial \Sigma$ and $\mu$ be the outward pointing unit normal vector field of $\partial\Sigma$ in $\Sigma$. Assume that $S_{ij}$ is a smooth time-varying symmetric tensor field on $\Sigma$ satisfying
	\begin{equation*}
	\frac{\partial}{\partial t}S_{ij}=a^{k\ell}\nabla_k\nabla_\ell S_{ij}+b^k\nabla_kS_{ij}+N_{ij}
	\end{equation*}
	on $\Sigma\times [0,T]$, where the coefficients $a^{k\ell}$ and $b^k$ are smooth, $\nabla$ is a smooth symmetric connection, and $(a^{k\ell})$ is positive definite everywhere. Suppose that
	\begin{equation}\label{conditon-MP}
	N_{ij}\xi^i\xi^j+\sup_{\Gamma}2a^{k\ell}\left(2\Gamma^p_k\nabla_\ell S_{ip}\xi^i-\Gamma_k^p\Gamma_\ell^qS_{pq}\right)\geq 0\quad \mathrm{on}~\Sigma\times (0,T],
	\end{equation}
	\begin{equation}\label{eq-bou}
	\left(\nabla_\mu S_{ij}\right)\xi^i\xi^j\geq 0\qquad \mathrm{on}~ \partial\Sigma\times (0,T]
	\end{equation}
	whenever $S_{ij}\geq 0$ and $S_{ij}\xi^i=0$. If $S_{ij}\geq 0$ everywhere on $\Sigma\times\{0\}$, then it remains so on $\Sigma\times[0,T]$.
\end{thm}
The tensor maximum principle is a generalization of the scalar maximum principle so that it applies to symmetric $2$-tensors. This was introduced firstly by Hamilton \cite[Theorem 9.1]{Hamilton1982} on compact manifolds to prove that the positivity of Ricci curvature is preserved along $3$-dimensional Ricci flow. Since then, it has been a powerful tool in the geometric flows and has been applied to find curvature preserving conditions along many kinds of flows. Hamilton's tensor maximum principle was refined by Andrews \cite[Theorem 3.1]{And2007} by weakening the null eigenvector condition $N_{ij}\xi^i\xi^j\geq 0$ to the condition \eqref{conditon-MP}. The tensor maximum principle on compact manifolds with proper Neumann boundary condition was also obtained by Stahl \cite[Theorem 3.3]{Stahl1996-2} under Hamilton's null eigenvector condition, in order to study the mean curvature flow with free boundary. Our Theorem \ref{s1:thm-max principle} can be viewed as a generalization of the tensor maximum principle of Hamilton \cite[Theorem 9.1]{Hamilton1982}, Andrews \cite[Theorem 3.1]{And2007} and Stahl \cite[Theorem 3.3]{Stahl1996-2}.

\begin{rem}
	Besides the null eigenvector condition $N_{ij}\xi^i\xi^j\geq 0$  and the boundary condition \eqref{eq-bou}, Stahl's tensor maximum principle in \cite[Theorem 3.3]{Stahl1996-2} needs an extra condition $\langle b,\mu\rangle\geq0$  on the boundary $\partial\Sigma$. Theorem \ref{s1:thm-max principle} refines Stahl's result by removing this condition and weakening the null eigenvector condtion to \eqref{conditon-MP}.
\end{rem}

The tensor maximum principle in Theorem \ref{s1:thm-max principle} is sufficient for us to prove that the strict convexity is preserved along the flow \eqref{s1:BGL-flow}. With the help of the strict convexity, we establish the convergence result of the flow \eqref{s1:BGL-flow} for strictly convex hypersurfaces with $\theta$-capillary boundary in the half-space.
\begin{thm}\label{s1:thm-convergence}
For $n\geq 2$, let $\Sigma\subset \-{\mathbb R}_{+}^{n+1}$ be a strictly convex hypersurface with capillary boundary supported on $\partial\-{\mathbb R}_{+}^{n+1}$ at a contact angle $\theta\in (0,\frac{\pi}{2}]$, which is given by the embedding $x_0:M \ra \-{\mathbb R}_{+}^{n+1}$ of a compact manifold $M$ with non-empty boundary in $\-{\mathbb R}_{+}^{n+1}$. Then for each $k=1,\cdots,n$, there exists a unique smooth solution $x:M\times [0,\infty)\ra \-{\mathbb R}_{+}^{n+1}$ to the flow \eqref{s1:BGL-flow}, which remains to be strictly convex and exists for all time $t\in [0,\infty)$. Moreover, $x(\cdot,t)$ converges to $x_\infty(\cdot)$ in the $C^\infty$-topology as $t\ra \infty$, and the limit $x_\infty:M\ra \-{\mathbb R}_+^{n+1}$ is a spherical cap.
\end{thm}

As an application of Theorem \ref{s1:thm-convergence}, we establish a complete family of Alexandrov-Fenchel type inequalities for convex capillary boundary hypersurfaces in the half-space.
\begin{thm}\label{s1:thm-AF-ineq}
	For $n\geq 2$, let $\Sigma\subset \-{\mathbb R}^{n+1}_{+}$ be a convex hypersurface with capillary boundary supported on $\partial\-{\mathbb R}_{+}^{n+1}$ at a contact angle $\theta\in (0,\frac{\pi}{2}]$. Then there holds
	\begin{align}\label{s1:AF-inequality-capillary}
	\frac{\mathcal{V}_{k,\t}(\widehat{\Sigma})}{|\mathbb B_\t^{n+1}|} \geq \(\frac{\mathcal{V}_{\ell,\t}(\widehat{\Sigma})}{|\mathbb B_\t^{n+1}|}\)^\frac{n+1-k}{n+1-\ell}, \quad 0\leq \ell<k\leq n.
	\end{align}
	Equality holds in \eqref{s1:AF-inequality-capillary} if and only if $\Sigma$ is a spherical cap  in the half-space with $\t$-capillary boundary.
\end{thm}
\begin{rem}
The case $k=n$ and $0\leq \ell<n$ of \eqref{s1:AF-inequality-capillary} was proved in Wang, Weng and Xia \cite[Theorem 1.2]{Wang-Weng-Xia2022}. Our theorem confirms the Conjecture \ref{s1.Ques} for convex capillary hypersurfaces in the half-space with contact angle $\t\in (0,\frac{\pi}{2}]$.
\end{rem}
\begin{rem}
For convex capillary hypersurfaces in the unit Euclidean ball $\mathbb{B}^{n+1}$, Scheuer-Wang-Xia \cite{Scheuer-Wang-Xia2018} and Weng-Xia \cite{WengX21} introduced the suitable quermassintegrals and established the Alexandrov-Fenchel type inequalities for the highest order $k=n$ and $\ell<n$.
\end{rem}

When $k=2$ and $\ell=1$, \eqref{s1:AF-inequality-capillary} implies the following Minkowski type inequality for convex capillary hypersurfaces with boundary in $\-{\mathbb R}^{n+1}_{+}$.
\begin{cor}\label{s1:cor-Minkowski-ineq}
	Let $\Sigma\subset \-{\mathbb R}^{n+1}_{+}, n\geq 2, $ be a convex hypersurface with capillary boundary supported on $\partial\-{\mathbb R}_{+}^{n+1}$ at a contact angle $\theta\in (0,\frac{\pi}{2}]$. Then there holds
	\begin{align}\label{s1.Minkowski-type-ineq}
	\int_{\Sigma} H dA \geq n(n+1)^\frac{1}{n}|\mathbb B_{\t}^{n+1}|^\frac{1}{n}\(|\Sigma|-\cos\t|\widehat{\partial \Sigma}|\)^\frac{n-1}{n}+\cos\t \sin\t|\partial \Sigma|.
	\end{align}
	Equality holds if and only if $\Sigma$ is a spherical cap  in the half-space with $\t$-capillary boundary.
\end{cor}

\begin{rem}
Recently, Wang-Weng-Xia \cite{WWX23} proved \eqref{s1.Minkowski-type-ineq} for star-shaped and mean convex capillary hypersurfaces for the whole range of contact
angle $\theta\in (0,\pi)$.
\end{rem}

The paper is organized as follows: In \S \ref{sec:2}, we collect some preliminaries on convex capillary hypersurfaces in the half-space and the elementary symmetric polynomials.  In \S \ref{sec:3}, we derive the evolution equations along the flow \eqref{s1:BGL-flow}. In \S \ref{sec:4}, we prove the generalized tensor maximum principle, i.e., Theorem \ref{s1:thm-max principle}. This will be applied in \S \ref{sec:5} to prove the preservation of strict convexity along the flow \eqref{s1:BGL-flow}. The long time existence and convergence of the flow \eqref{s1:BGL-flow} in Theorem \ref{s1:thm-convergence} will be proved in \S \ref{sec:6}, and as an application we obtain the Alexandrov-Fenchel type inequalities  in Theorem \ref{s1:thm-AF-ineq}.

\section{Preliminaries}\label{sec:2}

In this section, we collect some preliminaries on the capillary hypersurfaces in the half-space, the spherical caps and some basic properties on the elementary symmetric polynomials.
\subsection{Capillary hypersurfaces in the half-space}
Let $M$ be a compact orientable smooth manifold of dimension $n$ with boundary $\partial M$, and let $x:M\ra \-{\mathbb R}^{n+1}_{+}$ be a properly embedded smooth hypersurface. In particular,
$$
x(\operatorname{int}(M))\subset \mathbb R^{n+1}_{+} \quad \text{and}\quad  x(\partial M)\subset \partial \-{\mathbb R}_{+}^{n+1}.
$$
Denote $\Sigma=x(M)$ and $\partial \Sigma=x(\partial M)$. Let $\nu$ be the unit outward normal vector of $\Sigma$ in $\overline{\mathbb R}^{n+1}_{+}$, and $\-N$ the unit outward nomral of $\partial \-{\mathbb R}_{+}^{n+1}$ in $\-{\mathbb R}_{+}^{n+1}$. Let $\mu$ be the unit outward co-normal of $\partial \Sigma$ in $\Sigma$ and $\-\nu$ be the unit normal of $\partial \Sigma$ in $\partial \-{\mathbb R}^{n+1}_{+}$ such that $\{\nu,\mu\}$ and $\{\-\nu,\-N\}$ have the same orientation in the normal bundle of $\partial \Sigma \subset \-{\mathbb R}^{n+1}_{+}$. The contact angle $\t$ between the hypersurface $\Sigma$ and the support hyperplane $\partial \-{\mathbb R}^{n+1}_{+}$ is defined by $\langle \nu,\-N\rangle=\cos(\pi-\t)$. See Figure \ref{fig1} below. It follows that on the boundary $\partial \Sigma$, there holds
\begin{align}\label{s2:relation-normal}
\left\{\begin{aligned}\-N= &\sin\t \mu-\cos \t \nu, \\
\-\nu=&\cos \t \mu+\sin \t \nu.
\end{aligned}\right.
\end{align}
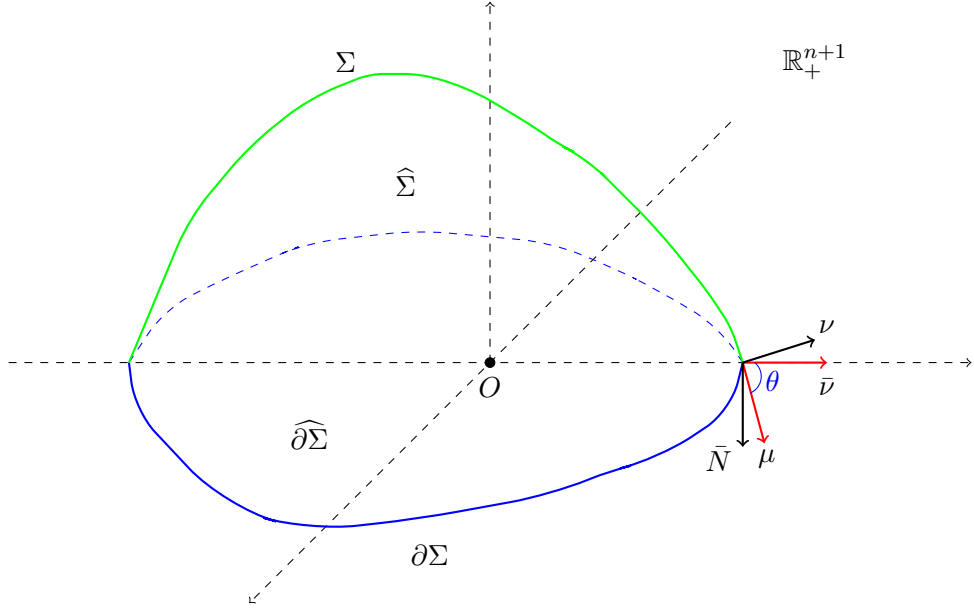
\begin{figure}[h]
	\begin{tikzpicture}[scale=1.6]
	\draw[dashed][->](-2,0)--(6,0);
	\draw[dashed][->](4,2)--(0,-2);
	\draw[dashed][->](2,0)--(2,3);
	\filldraw(2,0) circle(0.04);
	\node at (2,-0.2){$O$}[scale=0.5];
	\draw[thick][blue](-1,0)[rounded corners=11pt]--(-0.95,-0.4)
	[rounded corners=20pt]--(-0.2,-1.2)[rounded corners=18pt]--(0.5,-1.4)
	[rounded corners=15pt]--(1.4,-1.3)
	[rounded corners=15pt]--(2.5,-1.1)
	[rounded corners=8pt]--(3.0,-0.9)
	[rounded corners=15pt]--(3.4,-0.8)
	[rounded corners=11pt]--(4,-0.4)
	[rounded corners=11pt]--(4.1,0);
	\draw[dashed][blue](-1,0)[rounded corners=11pt]--(-0.7,0.45)
	[rounded corners=20pt]--(0,0.8)[rounded corners=11pt]--(0.5,1.0)
	[rounded corners=11pt]--(1.4,1.1)
	[rounded corners=11pt]--(2.5,1.0)
	[rounded corners=11pt]--(3.0,0.8)
	[rounded corners=11pt]--(3.4,0.6)
	[rounded corners=11pt]--(3.9,0.3)
	[rounded corners=11pt]--(4.1,0);
	\draw[thick][green](-1,0)[rounded corners=11pt]--(-0.5,1.2)
	[rounded corners=11pt]--(0,1.8)
	[rounded corners=11pt]--(0.5,2.2)
	[rounded corners=5pt]--(1,2.4)
	[rounded corners=5pt]--(1.4,2.4)
	[rounded corners=11pt]--(1.8,2.3)
	[rounded corners=5pt]--(2.3,2)
	[rounded corners=5pt]--(2.6,1.8)
	[rounded corners=10pt]--(2.8,1.7)
	[rounded corners=5pt]--(3.0,1.5)
	[rounded corners=11pt]--(3.4,1.1)
	[rounded corners=5pt]--(3.8,0.6)
	[rounded corners=5pt]--(4.0,0.3)
	[rounded corners=11pt]--(4.1,0);
	\draw[->][red][thick](4.1,0)--(4.8,0);
	\draw[->][red][thick](4.1,0)--(4.28,-0.67);
	\draw[->][thick](4.1,0)--(4.7,0.19);
	\draw[->][thick](4.1,0)--(4.1,-0.7);
	\node at (4.7,2.5){$\mathbb{R}_{+}^{n+1}$};
	\node at (1.5,-1.6){$\partial\Sigma$};
	\node at (0.5,-0.6){$\widehat{\partial\Sigma}$};
	\node at (1.3,1.5) {$\widehat{\Sigma}$};
	\node at (0.8,2.5){$\Sigma$};
	\node at (3.9,-0.8){$\bar{N}$};
	\node at (4.3,-0.8){$\mu$};
	\node at (4.8,-0.2){$\bar{\nu}$};
	\node at (4.8,0.3){$\nu$};
	\draw[blue](4.2,0) arc(50:-67:0.15);
	\node[blue] at (4.35,-0.15){$\theta$};
	\end{tikzpicture}
	\caption{Capillary hypersurface $\Sigma$ in the half-space with contact angle $\theta$}\label{fig1}
\end{figure}

We use $D$ to denote the Levi-Civita connection of $\-{\mathbb R}^{n+1}_{+}$ with respect to the Euclidean metric $\d$, and $\nabla$ the Levi-Civita connection on $\Sigma$ with respect to the induced metric $g$, respectively. The second fundamental form $h$ of $\Sigma$ in $\-{\mathbb R}^{n+1}_{+}$ is defined by
\begin{align*}
D_X Y=\nabla_X Y-h(X,Y)\nu.
\end{align*}
The principal curvatures $\k=(\k_1,\cdots,\k_n)$ of the hypersurface $\Sigma$ are defined as the eigenvalues of the Weingarten matrix $\mathcal{W}=(h_i^j)=(g^{kj}h_{ik})$.

Note that $\partial \Sigma\subset \partial \-{\mathbb R}^{n+1}_{+}$ can be viewed as a smooth closed hypersurface in $\mathbb R^n$, which encloses a bounded domain $\widehat{\partial \Sigma}$ in $\mathbb R^n$. Since $\langle \-\nu, \-N\circ x\rangle=0$, the second fundamental form of $\partial \Sigma$ in $\mathbb R^n$ is given by
\begin{align*}
\widehat{h}(X,Y):=-\langle \nabla^{\mathbb R^n}_{X}Y,\-\nu\rangle=-\langle D_XY,\bar{\nu}\rangle, \quad X,Y\in T(\partial \Sigma).
\end{align*}
Similarly, due to $\langle \nu,\mu\rangle=0$, the second fundamental form of $\partial\Sigma$ in $\Sigma$ is given by
\begin{align*}
\tilde{h}(X,Y):=-\langle \nabla_{X}Y,\mu\rangle=-\langle D_XY,\mu\rangle, \quad X,Y\in T(\partial \Sigma).
\end{align*}

The following property of capillary hypersurfaces in the half-space is given in \cite[Proposition 2.3]{Wang-Weng-Xia2022}.
\begin{prop}[\cite{Wang-Weng-Xia2022}]\label{s2:prop-second-fundamental-form}
	Let $\Sigma\subset \-{\mathbb R}^{n+1}_{+}$ be a capillary hypersurface in the half-space with contact angle $\t\in(0,\frac{\pi}{2}]$. Let $\{e_\a\}_{\a=1}^{n-1}$ be an orthonormal frame of $\partial \Sigma$. Then along $\partial \Sigma$, we have
	\begin{enumerate}[(1)]
		\item $\mu$ is a principal direction of $\Sigma$, that is, $h_{\mu\a}=h(\mu,e_\a)=0$.
	    \item $h_{\a\b}=\sin \t \widehat{h}_{\a\b}$.
	    \item $\tilde{h}_{\a\b}=\cos\t \widehat{h}_{\a\b}=\cot\t h_{\a\b}$.
	    \item $\nabla_\mu h_{\a\b}=\tilde{h}_{\b\g}(h_{\mu\mu}\d_{\a\g}-h_{\a\g})$.
	\end{enumerate}
\end{prop}

\subsection{Spherical caps}
Let $e:=-e_{n+1}=(0,\cdots,0,-1)$. The family of spherical caps lying entirely in $\-{\mathbb R}_{+}^{n+1}$ and intersecting with $\partial \-{\mathbb R}^{n+1}=\mathbb R^n$ at a constant contact angle $\t\in (0,\pi)$ are given by
\begin{align}\label{s2:spherical-caps}
C_{r,\t}(e):=\left\{  x\in \-{\mathbb R}^{n+1}_{+} ~|~|x-r\cos\t e|=r \right\},\quad r>0,
\end{align}
which is part of the round sphere with radius $r$ and centered at $r\cos\t e$. See Figure \ref{fig2}. We also denote $C_{r,\theta}=C_{r,\theta}(e)$ for simplicity when there is no confusion. For unit radius $r=1$, we just write $\mathbb{S}_\theta^n=C_{1,\theta}$ and $\mathbb{B}^{n+1}_\theta=\widehat{C_{1,\theta}}$.
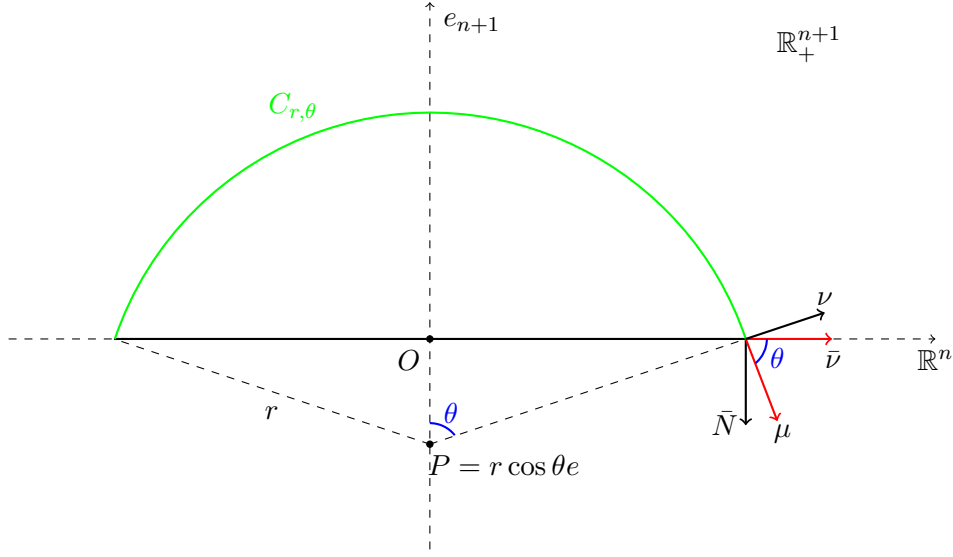
\begin{figure}[h]
	\begin{tikzpicture}[scale=1.4]
	\draw[dashed][->](-4,0)--(4.8,0);
	\draw[thick](-3,0)--(3,0);
	\draw[thick][green](3,0) arc (18.66:161.34:3.162);
	\draw[dashed][->](0,-2)--(0,3.2);
	\filldraw(0,-1) circle(0.03);
	\draw[dashed](0,-1)--(3,0);
	\draw[dashed](0,-1)--(-3,0);
	\draw[thick][red][->](3,0)--(3.82,0);
	\draw[thick][->](3,0)--(3,-0.82);
	\draw[thick][red][->](3,0)--(3.3,-0.78);
	\draw[thick][->](3,0)--(3.75,0.25);
	\node at (0.7,-1.2) {$P=r\cos\theta e$};
	\node at (2.8,-0.8) {$\bar{N}$};
	\node at (3.35,-0.9) {$\mu$};
	\node at (3.83,-0.2) {$\bar{\nu}$};
	\node at (3.75,0.38) {$\nu$};
	\draw[blue][thick](3.2,0) arc(0:-50:0.3);
	\node[blue] at (3.3,-0.18) {$\theta$};
	\node at (4.8,-0.2) {$\mathbb{R}^n$};
	\node at (0.4,3.0) {$e_{n+1}$};
	\filldraw(0,0) circle(0.03);
	\node at (-0.2,-0.2) {$O$};
	\node at (3.6,2.8) {$\mathbb{R}^{n+1}_{+}$};
	\node[green] at (-1.3,2.2) {$C_{r,\theta}$};
	\draw[blue][thick](0,-0.8) arc(90:38:0.3);
	\node[blue] at (0.2,-0.7) {$\theta$};
	\node at (-1.5,-0.7) {$r$};
	\end{tikzpicture}
	\caption{Spherical cap $C_{r,\theta}(e)$ in the half-space}\label{fig2}
\end{figure}

The spherical cap $C_{r,\t}$ satisfies
\begin{align*}
1+\cos\t \langle \nu,e\rangle-\frac{1}{r}\langle x,\nu\rangle=0
\end{align*}
and thus is stationary along the flow \eqref{s1:BGL-flow}. Moreover, it is easy to check that
\begin{align}\label{s2:AF-equality}
\mathcal{V}_{k,\t}(\widehat{C_{r,\t}})=|\mathbb{B}^{n+1}_\theta| r^{n+1-k}, \quad 0\leq k\leq n.
\end{align}
Therefore, $\widehat{C_{r,\t}}$ achieves the equality in the Alexandrov-Fenchel inequalities \eqref{s1:AF-inequality-capillary}.

\subsection{Elementary symmetric polynomials}
For $\kappa=(\kappa_1,\cdots,\kappa_n)\in \mathbb{R}^n$, the $k$th elementary symmetric polynomial function is defined as
\begin{equation*}
  \sigma_k(\kappa)=\sum_{1\leq i_1<\cdots <i_k\leq n}\kappa_{i_1}\cdots \kappa_{i_k},\qquad k=1,\cdots,n
\end{equation*}
and its normalization is denoted by $H_k(\k)=\binom{n}{k}^{-1}\s_k(\k)$. We also take $\s_0(\k)=H_0(\k)=1$ by convention. This definition can be extended to symmetric matrices. Let $W\in \mathrm{Sym}(n)$ be an $n\times n$ symmetric matrix. Denote $\kappa=\kappa(W)$ the eigenvalues of $W$. Set $\sigma_k(W)=\sigma_k(\kappa(W))$. We have
\begin{align*}
  \sigma_k(W) =& \frac{1}{k!}\delta_{i_1,\cdots,i_k}^{j_1,\cdots,j_k}W_{i_1j_1}\cdots W_{i_kj_k},\quad k=1,\cdots,n,
\end{align*}
where $\delta_{i_1,\cdots,i_k}^{j_1,\cdots,j_k}$ is the generalized Kronecker delta.

Let $\G_k^{+}$ be the Garding cone defined by
$$
\G_{k}^{+}:=\{\k\in\mathbb R^n|\sigma_j(\k)>0,\forall 1\leq j\leq k\}
$$
and $\Gamma_+$ be the positive cone
\begin{equation*}
  \Gamma_+:=\{\k\in\mathbb R^n|\kappa_i>0,\forall 1\leq i\leq n\}.
\end{equation*}
We have the following Newton-MacLaurin inequalities (see e.g.\cite[Lemma 2.5]{guan2014}).
\begin{prop}
	For $1\leq \ell<k \leq n$,  we have
	\begin{align}\label{s2:NM-ineq}
	H_\ell H_{k-1} \geq H_{\ell-1}H_k, \quad \forall \k\in \G_k^{+}.
	\end{align}
	Equality holds if and only if $\k=\l(1,\cdots,1)$ for any constant $\l>0$.
\end{prop}

In this paper, we will focus on the function $F={H_\ell}/{H_{\ell-1}}$, where $1\leq \ell\leq n$. Sometimes, during the calculation of some evolution equations, we write $F=F(W)=f(\kappa)$ for $\kappa=\kappa(W)$, where we view $F$ as a smooth function of the symmetric matrix $W$ and $f$ as a smooth function of the eigenvalues $\kappa$ of $W$. Denote $\dot{F}^{k\ell}, \ddot{F}^{k\ell,pq}$ and $\dot{f}^i, \ddot{f}^{ij}$ as the first and second derivatives of $F$ and $f$ with respect to their arguments. The derivatives of $f$ and $F$ are related in the following way: if $W$ is diagonal, then
\begin{equation*}
  \dot{F}^{k\ell}(W)=\dot{f}^k(\kappa)\delta_{k\ell}.
\end{equation*}
At any diagonal $W$ with distinct eigenvalues, the second derivatives $\ddot{F}$ of $F$ in direction $B\in \mathrm{Sym}(n)$ are given in terms of $\dot{f}$ and $\ddot{f}$ by  (see \cite{And2007}):
\begin{equation}\label{s2:F-ddt}
  \ddot{F}^{ij,k\ell}(W)B_{ij}B_{k\ell}=\sum_{i,k}\ddot{f}^{ik}(\kappa)B_{ii}B_{kk}+2\sum_{i>k}\frac{\dot{f}^i(\kappa)-\dot{f}^k(\kappa)}{\kappa_i-\kappa_k}B_{ik}^2.
\end{equation}
This formula makes sense as a limit in the case of any repeated values of $\kappa_i$. In what follows, we will drop the arguments when $F$ and $f$, and their
derivatives, are evaluated at $W$ or $\kappa=\kappa(W)$.

\begin{lem}\label{s2.lem3}
Let $F={H_\ell}/{H_{\ell-1}}$, where $1\leq \ell\leq n$.
\begin{enumerate}
  \item $F$ is homogeneous of degree one, so $F=\dot{F}^{k\ell}W_{k\ell}$.
  \item $F$ is concave in $\Gamma_{+}$, and so $(\dot{f}^i-\dot{f}^j)(\kappa_i-\kappa_j)\leq 0$ for $i\neq j$. If $W$ is diagonal, we also have $(\dot{F}^{ii}-\dot{F}^{jj})(W_{ii}-W_{jj})\leq 0$ for $i\neq j$. Moreover, $\sum_i\dot{f}^i\geq 1$ on $\Gamma_+$.
  \item If $\k\in \G_{+}$, then there holds
\begin{align}\label{s2:key-inequality}
F^2\leq \dot{F}^{k\ell}(W^2)_{k\ell}\leq (n-\ell+1)F^2.
\end{align}
where $(W^2)_{k\ell}=W_{kp}W_{p\ell}$.
\end{enumerate}
\end{lem}
\begin{proof}
These properties are well known. See \cite{guan2014} and \cite{Andrews-McCoy-Zheng}. We include a proof for (iii) for convenience of the readers.

Firstly, we have
$$
F=\frac{\s_{\ell}(\k)}{\s_{\ell-1}(\k)}\frac{\s_{\ell-1}(I)}{\s_{\ell}(I)},
$$
where $I=(1,\cdots,1)\in \mathbb{R}^n$. Using the formula (see \cite[Proposition 2.2]{guan2014})
\begin{equation*}
  \sum_{i=1}^n \dot{\sigma}_\ell^i(\kappa)\k_i^2=\s_1(\k)\s_\ell(\k)-(\ell+1)\s_{\ell+1}(\k),
\end{equation*}
we get
\begin{align*}
\dot{F}^{k\ell}(h^2)_{k\ell}=&\sum_{i=1}^{n}\dot{f}^i\k_i^2\\
=&\frac{\s_{\ell-1}(I)}{\s_{\ell}(I)} \frac{\dot{\s}_{\ell}^i(\kappa)\k_i^2\s_{\ell-1}(\k)-\s_{\ell}(\k)\dot{\s}^i_{\ell-1}(\kappa)\k_i^2}{\s_{\ell-1}^2(\k)}\\
=&\frac{\s_{\ell-1}(I)}{\s_{\ell}(I)} \frac{\ell\s_{\ell}^2(\k)-(\ell+1)\s_{\ell+1}(\k)\s_{\ell-1}(\k)}{\s_{\ell-1}^2(\k)}\\
=&\frac{(n-\ell+1)H_{\ell}^2-(n-\ell)H_{\ell-1}H_{\ell+1}}{H_{\ell-1}^2}.
\end{align*}
Then the inequality \eqref{s2:key-inequality} follows from the Newton-MacLaurin inequality
$$
H_{\ell}^2 \geq H_{\ell-1}H_{\ell+1}
$$
and $H_{\ell-1}H_{\ell+1}\geq 0$ since $\k\in \G_{+}$.
\end{proof}

The following inverse-concave property on $F$ can be found in \cite{And2007,AW18,WeiX15}.
\begin{lem}
Let $F={H_\ell}/{H_{\ell-1}}$, where $1\leq \ell\leq n$. If $\k=\kappa(W)\in \G_{+}$, then $F(W)=f(\kappa)$ is inverse concave. That is, its dual function
\begin{equation*}
  f_*(\tau):=\frac{1}{f(\frac{1}{\tau_1},\cdots,\frac{1}{\tau_n})}
\end{equation*}
is concave, where $\tau=(\tau_1,\cdots,\tau_n)$ with $\tau_i>0$. Moreover, $f$ satisfies
 \begin{equation}\label{s2.d2F}
   \sum_{k,\ell=1}^n\ddot{f}^{k\ell}y_ky_\ell+2\sum_{k=1}^n\frac {\dot{f}^k}{\kappa_k}y_k^2~\geq ~2f^{-1}(\sum_{k=1}^n\dot{f}^ky_k)^2
 \end{equation}
 for any $y=(y_1,\cdots,y_n)\in \mathbb{R}^n$, and
\begin{equation}\label{s2:inv-conc}
\frac{\dot{f}^k-\dot{f}^\ell}{\kappa_k-\kappa_\ell}+\frac{\dot{f}^k}{\kappa_\ell}+\frac{\dot{f}^\ell}{\kappa_k}\geq~0,\quad \forall~k\neq \ell.
\end{equation}
\end{lem}

\section{Evolution equations}\label{sec:3}
In this section, we derive the evolution equations along the flow \eqref{s1:BGL-flow}. We first recall some basic evolution equations along a general flow:
\begin{align}\label{s3.flow}
(\partial_tx)^\bot=\mathcal{F}\nu,
\end{align}
for a smooth function $\mathcal{F}$. If $\Sigma_t\subset \-{\mathbb R}_+^{n+1}$ is a family of smooth $\t$-capillary hypersurfaces, given by the embedding $x(\cdot,t): M\to \-{\mathbb R}_+^{n+1}$, evolving along \eqref{s3.flow}, then the tangential component $V:=(\partial_t x)^\top$ must satisfy
\begin{align*}
V|_{\partial \Sigma_t}=\mathcal{F}\cot \t \mu+\widetilde{V},
\end{align*}
where $\widetilde{V}\in T (\partial \Sigma_t)$. This follows from the fact that $x(\cdot,t)|_{\partial M}$ is contained in $\partial \overline{\mathbb R}^{n+1}_+$ and hence
\begin{align*}
(\mathcal{F}\nu+V)|_{\partial \Sigma_t}=\partial_t x|_{\partial M} \in T\mathbb R^n.
\end{align*}
By the relation \eqref{s2:relation-normal} along $\partial \Sigma_t$, we get
\begin{align*}
\nu=\frac{1}{\sin\t}\-\nu-\cot\t \mu.
\end{align*}
It follows that $(V-\mathcal{F}\cot\t\mu)|_{\partial\Sigma_t} \in T\mathbb R^n \cap T\Sigma_t=T(\partial \Sigma_t)$. Up to a tangential diffeomorphism of $\partial M$, we may assume that $\widetilde{V}=0$. For simplicity, we always assume that
\begin{align}\label{s3:tangential-relation}
V|_{\partial \Sigma_t}=\mathcal{F}\cot\t \mu.
\end{align}
Therefore, the flow \eqref{s3.flow} is equivalent to the flow
\begin{align}\label{s2:general-flow}
\partial_tx=\mathcal{F}\nu+V,
\end{align}
where $V\in T\Sigma_t$ is a tangential vector field of $\Sigma_t$ satisfying \eqref{s3:tangential-relation} on the boundary $\partial\Sigma_t$.

Let $g_{ij}$, $\nu$, $h_{ij}$ and $H$ be the induced metric, the unit outward normal, the second fundamental form and the mean curvature of the solution $\Sigma_t$ of the flow \eqref{s2:general-flow}, and $F=F(h_i^j)$ be a smooth function of the Weingarten matrix $(h_i^j)$ of $\Sigma_t$. We have the following general equations along the flow \eqref{s2:general-flow}, see \cite{HP99} for $V=0$ and  \cite[Proposition 2.7]{Wang-Weng-Xia2022} for a general $V$.
\begin{prop}\label{s3:prop-evol-general-flow}
Along the flow \eqref{s2:general-flow}, there holds
\begin{align}
\partial_t g_{ij}=&2\mathcal{F}h_{ij}+\nabla_i V_j+\nabla_j V_i, \label{s2:metric-general-f}\\
\partial_t \nu=&-\nabla \mathcal{F}+h(e_i,V)e_i, \label{s2:normal-general-f}\\
\partial_t h_{ij}=&-\nabla_i\nabla_j \mathcal{F}+\mathcal{F} (h^2)_i^j+\nabla_V h_{ij}+h_j^k \nabla_i V_k+h_i^k \nabla_j V_k,\label{s2:hij-general-f}\\
\partial_t H=&-\D \mathcal{F}-|A|^2 \mathcal{F}+\nabla_V H,\label{s2:H-general-f}\\
\partial_t F=&-\dot{F}^{k\ell}\nabla_k\nabla_\ell \mathcal{F}-\mathcal{F}\dot{F}^{k\ell} (h^2)_{k\ell}+\nabla_V F, \label{s2:F-general-f}
\end{align}
where $(h^2)_{ij}=h_{ik}h^k_j$ and $|A|^2=h^j_k h^k_j$.
\end{prop}

In the following,  we focus on the locally constrained inverse curvature flow \eqref{s1:BGL-flow}. Denote $F=H_{\ell}/H_{\ell-1}$ for $\ell=1,\cdots,n$ and
\begin{equation}\label{s3.F}
\mathcal{F}=\frac{1+\cos\t\langle \nu,e\rangle}{F}-\langle x,\nu\rangle.
\end{equation}
Then the flow equation \eqref{s1:BGL-flow} is equivalent to
\begin{align}\label{s2:BGL-flow}
\partial_tx=\mathcal{F}\nu+V,
\end{align}
where $\mathcal{F}$ is given by \eqref{s3.F} and $V\in T\Sigma_t$ is a tangential vector field of $\Sigma_t$ satisfying \eqref{s3:tangential-relation}. We note that the capillary boundary condition $\langle \overline{N}\circ x, \nu\rangle =\cos(\pi-\theta)$ implies that $\mathcal{F}$ satisfies
\begin{equation}\label{s3:DF}
  \nabla_\mu \mathcal{F}=\cot\theta h_{\mu\mu}\mathcal{F}
\end{equation}
on the boundary $\partial\Sigma_t$. This follows from taking derivatives of the boundary condition $\langle \overline{N}\circ x, \nu\rangle =\cos(\pi-\theta)$ in time $t$ and using \eqref{s2:relation-normal}, \eqref{s2:normal-general-f} and \eqref{s3:tangential-relation}.
\begin{prop}
	Along the flow \eqref{s2:BGL-flow}, we have the following evolution equations:
	\begin{enumerate}[(i)]
		\item The induced metric $g_{ij}$ evolves by
		\begin{align}\label{s3:evol-metric}
		\partial_tg_{ij}=2\(\frac{1+\cos\t\langle \nu,e\rangle}{F}-\langle x,\nu\rangle \)h_{ij}+\nabla_i V_j+\nabla_j V_i.
		\end{align}
		\item The second fundamental form $h_{ij}$ evolves by
		\begin{align}\label{s3:evol-2nd-fundamental-form}
		\partial_th_{ij}=&\frac{1+\cos\t\langle \nu,e\rangle}{F^2}\dot{F}^{k\ell}\nabla_k \nabla_\ell h_{ij}+\left\langle V+x-\frac{\cos\t}{F}e,\nabla h_{ij}\right\rangle\nonumber\\
&+\frac{1+\cos\t\langle \nu,e\rangle}{F^2}\ddot{F}^{k\ell,pq}\nabla_i h_{k\ell}\nabla_j h_{pq} \nonumber\\
		&-\frac{2(1+\cos\t\langle \nu,e\rangle)}{F^3}\nabla_iF\nabla_jF\nonumber\\
		&+\frac{\cos\theta}{F^2}\left(h_i^k\langle e_k,e\rangle \nabla_jF+h_j^k\langle e_k,e\rangle \nabla_iF\right)\nonumber\\
		&+\(\frac{\cos\t\langle \nu,e\rangle}{F}-2\langle x,\nu\rangle\) (h^2)_{ij}+h_j^k \nabla_i V_k+h_i^k\nabla_j V_k \nonumber\\
&+\(1+\frac{1+\cos\t\langle \nu,e\rangle}{F^2}\dot{F}^{k\ell}(h^2)_{k\ell}\) h_{ij}.
		\end{align}
		\item The curvature function $F$ evolves by
		\begin{align}\label{s3:evol-F}
	    \partial_t F=&\frac{1+\cos\t\langle \nu,e\rangle}{F^2}\dot{F}^{k\ell}\nabla_k\nabla_\ell F+\left\langle V+x-\frac{\cos\t}{F}e,\nabla F\right\rangle\nonumber\\
		&-\frac{2(1+\cos\t\langle \nu,e\rangle)}{F^3}\dot{F}^{k\ell}\nabla_k F\nabla_\ell F\nonumber\\
		&+\frac{2\cos\t}{F^2}\dot{F}^{k\ell}\nabla_k F h_\ell^j\langle e_j,e\rangle+F-\frac{1}{F}\dot{F}^{k\ell}(h^2)_{k\ell},
		\end{align}
		and
		\begin{align}\label{s3:bdry-F}
		\nabla_\mu F=0, \quad \text{on $\partial \Sigma_t$}.
		\end{align}
		\item The mean curvature $H$ evolves by
		\begin{align}\label{s3:evol-mean-curvature}
		\partial_t H=&\frac{1+\cos\t\langle \nu,e\rangle}{F^2}\dot{F}^{k\ell}\nabla_k \nabla_\ell H+\left\langle V+x-\frac{\cos\t}{F}e,\nabla H\right\rangle\nonumber\\
		&+\frac{1+\cos\t\langle \nu,e\rangle}{F^2}\ddot{F}^{k\ell,pq}\nabla_i h_{k\ell}\nabla_i h_{pq}\nonumber\\
&-\frac{2(1+\cos\t\langle \nu,e\rangle)}{F^3}|\nabla F|^2+\frac{2\cos\t}{F^2}\langle \nabla F,\nabla \langle \nu,e\rangle\rangle \nonumber\\
		&+\(\frac{1+\cos\t\langle \nu,e\rangle}{F^2}\dot{F}^{k\ell}(h^2)_{k\ell}+1\)H\nonumber\\
&-\frac{2+\cos\t\langle \nu,e\rangle}{F}|A|^2 .
		\end{align}
		If $\Sigma_t$ is a strictly convex solution to the flow \eqref{s2:BGL-flow}, then
		\begin{align}\label{s3:bdry-H}
		\nabla_\mu H\leq 0, \quad \text{on $\partial \Sigma_t$}.
		\end{align}
\item The support function $u=\langle x,\nu\rangle$ evolves by
\begin{align}\label{s3.evl-u}
  \partial_tu= & \frac{1+\cos\t\langle \nu,e\rangle}{F^2}\dot{F}^{k\ell}\nabla_k \nabla_\ell u +\left\langle V+x-\frac{\cos\t}{F}e,\nabla u\right\rangle\nonumber\\
  & +u\left(\frac{1+\cos\t\langle \nu,e\rangle}{F^2}\dot{F}^{ij}(h^2)_{ij}-1\right).
\end{align}
\item The function $\langle \nu,e\rangle $ evolves by
\begin{align}\label{s3.evl-nue}
  \partial_t\langle \nu,e\rangle = & \frac{1+\cos\t\langle \nu,e\rangle}{F^2}\dot{F}^{k\ell}\nabla_k \nabla_\ell \langle \nu,e\rangle  +\left\langle V+x-\frac{\cos\t}{F}e,\nabla \langle \nu,e\rangle \right\rangle\nonumber\\
  & +\frac{1+\cos\t\langle \nu,e\rangle}{F^2}\dot{F}^{ij}(h^2)_{ij}\langle \nu,e\rangle .
\end{align}
	\end{enumerate}
\end{prop}
\begin{proof}
	(i) The  equation \eqref{s3:evol-metric} follows from substituting \eqref{s3.F} into \eqref{s2:metric-general-f}.
	
	(ii) By the evolution equation \eqref{s2:hij-general-f}, we need to calculate
\begin{align*}
 -\nabla_i \nabla_j\mathcal{F}=& -\nabla_i \nabla_j\(\frac{1+\cos\t\langle \nu,e\rangle}{F}-\langle x,\nu\rangle\) \\
  = & \frac{1+\cos\t\langle \nu,e\rangle}{F^2}\nabla_i \nabla_jF-2\frac{1+\cos\t\langle \nu,e\rangle}{F^3}\nabla_iF \nabla_jF\\
  &+\frac{\cos\theta}{F^2}\left(\nabla_i\langle\nu,e\rangle \nabla_jF+\nabla_j\langle\nu,e\rangle \nabla_iF\right)\\
  &-\frac{\cos\theta}{F}\nabla_i \nabla_j\langle\nu,e\rangle+\nabla_i \nabla_j\langle x,\nu\rangle.
\end{align*}
Using the Weingarten and Codazzi equations, we have
\begin{align}
  \nabla_i\langle \nu,e\rangle= & h_i^k\langle e_k,e\rangle, \label{s3.dnue}\\
  \nabla_i\langle x,\nu\rangle =&\langle x,\nabla_i\nu\rangle=\langle x,e_k\rangle h_i^k\label{s3.du}\\
   \nabla_i\nabla_j\langle \nu,e\rangle= & \langle \nabla h_{ij},e\rangle -(h^2)_{ij}\langle \nu,e\rangle,\label{s3.d2nue} \\
   \nabla_i\nabla_j\langle x,\nu\rangle= & h_{ij}+\langle x, \nabla h_{ij}\rangle -(h^2)_{ij}\langle x,\nu\rangle.\label{s3.d2u}
\end{align}
Using the Codazzi equations, the Ricci identities and the Gauss equations, we obtain the following Simons' identity (see \cite{CM2011})
	\begin{align*}
	\nabla_i\nabla_j h_{k\ell}=&\nabla_i\nabla_k h_{\ell j}
	=\nabla_k\nabla_i h_{\ell j}+R_{m\ell ki}h_{mj}+R_{mjki}h_{\ell m} \nonumber\\
	=&\nabla_k\nabla_\ell h_{ij}+(h_{mk}h_{\ell i}-h_{k\ell}h_{mi})h_{mj}+(h_{ij}h_{mk}-h_{jk}h_{mi})h_{\ell m},
	\end{align*}
where $R_{m\ell ki}$ denotes the Riemannian curvature tensor of the metric $g_{ij}$ on $\Sigma_t$. This implies that
	\begin{align*}
	\nabla_i\nabla_j F=&\dot{F}^{k\ell}\nabla_i\nabla_j h_{k\ell}+\ddot{F}^{k\ell,pq}\nabla_i h_{k\ell}\nabla_j h_{pq} \nonumber\\
	=&\dot{F}^{k\ell}\nabla_k\nabla_\ell h_{ij}-\dot{F}^{k\ell}h_{k\ell}(h^2)_{ij}+\dot{F}^{k\ell}(h^2)_{k\ell}h_{ij}+\ddot{F}^{k\ell,pq}\nabla_i h_{k\ell}\nabla_j h_{pq} \nonumber\\
	=&\dot{F}^{k\ell}\nabla_k \nabla_\ell h_{ij}-F(h^2)_{ij}+\dot{F}^{k\ell}(h^2)_{k\ell}h_{ij}+\ddot{F}^{k\ell,pq}\nabla_i h_{k\ell}\nabla_j h_{pq},
	\end{align*}
	where we used $F=\dot{F}^{k\ell}h_{k\ell}$ due to the 1-homogeneity of $F$. Then
\begin{align}\label{s3.D2F}
 -\nabla_i \nabla_j\mathcal{F}  = & \frac{1+\cos\t\langle \nu,e\rangle}{F^2}\left(\dot{F}^{k\ell}\nabla_k \nabla_\ell h_{ij}+\ddot{F}^{k\ell,pq}\nabla_i h_{k\ell}\nabla_j h_{pq}\right)\nonumber\\
 &-\frac{1+\cos\t\langle \nu,e\rangle}{F}(h^2)_{ij}+\left(1+\frac{1+\cos\t\langle \nu,e\rangle}{F^2}\right)\dot{F}^{k\ell}(h^2)_{k\ell}h_{ij}\nonumber\\
 &-2\frac{1+\cos\t\langle \nu,e\rangle}{F^3}\nabla_iF \nabla_jF\nonumber\\
 &+\frac{\cos\theta}{F^2}\left(h_i^k\langle e_k,e\rangle \nabla_jF+h_j^k\langle e_k,e\rangle \nabla_iF\right)\nonumber\\
  &+\langle x-\frac{\cos\theta}{F}e, \nabla h_{ij}\rangle +\left(\frac{\cos\theta}{F}\langle \nu,e\rangle- \langle x,\nu\rangle\right)(h^2)_{ij}.
\end{align}
Substituting \eqref{s3.D2F} and \eqref{s3.F} into \eqref{s2:hij-general-f} and rearranging the terms,  we obtain \eqref{s3:evol-2nd-fundamental-form}.
	
	(iii) The calculation for the evolution equation \eqref{s3:evol-F} of $F$ and the boundary condition \eqref{s3:bdry-F} is essentially the same as that in \cite[Proposition 4.3]{Wang-Weng-Xia2022} for $F=H_n/{H_{n-1}}$.  In fact, \eqref{s3:evol-F} can be obtained by using the two equations \eqref{s3:evol-metric} and \eqref{s3:evol-2nd-fundamental-form}, and \eqref{s3:bdry-F} follows from \eqref{s3:DF} and the calculation for $\nabla_\mu\langle x,\nu\rangle$ and $\nabla_\mu\langle \nu,e\rangle$.  So we omit the details.

	(iv) The calculation for the equation \eqref{s3:evol-mean-curvature} is similar as that in  \cite[Proposition 4.4]{Wang-Weng-Xia2022} for the case $F=H_n/{H_{n-1}}$. The only difference is that the fact $F^2=\dot{F}^{k\ell}(h^2)_{k\ell}$ is used, which is due to that $F=H_n/{H_{n-1}}$, while this identity is not true for general $F=H_\ell/H_{\ell-1}$. We include a proof here for the completeness.

Taking the trace of \eqref{s3.D2F}, we have
\begin{align}\label{s3.D2H}
 -\Delta \mathcal{F}  = & \frac{1+\cos\t\langle \nu,e\rangle}{F^2}\left(\dot{F}^{k\ell}\nabla_k \nabla_\ell H+\ddot{F}^{k\ell,pq}\nabla_i h_{k\ell}\nabla_i h_{pq}\right)\nonumber\\
 &-\frac{1+\cos\t\langle \nu,e\rangle}{F}|A|^2+\left(1+\frac{1+\cos\t\langle \nu,e\rangle}{F^2}\right)\dot{F}^{k\ell}(h^2)_{k\ell}H\nonumber\\
 &-2\frac{1+\cos\t\langle \nu,e\rangle}{F^3}|\nabla F|^2+\frac{2\cos\theta}{F^2}\langle \nabla\langle\nu,e\rangle, \nabla F\rangle \nonumber\\
  &+\langle x-\frac{\cos\theta}{F}e, \nabla H\rangle +\left(\frac{\cos\theta}{F}\langle \nu,e\rangle- \langle x,\nu\rangle\right)|A|^2.
\end{align}
Substituting \eqref{s3.D2H} and \eqref{s3.F} into \eqref{s2:H-general-f}, we obtain \eqref{s3:evol-mean-curvature}.

	To show the boundary condition \eqref{s3:bdry-H}, we choose an orthonormal frame $\{e_\alpha\}_{\alpha=1}^{n-1}$ of $T(\partial\Sigma_t)$ such that $\{\mu,e_1,\cdots,e_{n-1}\}$ forms an orthonormal frame of $T\Sigma_t$. By the equation \eqref{s3:bdry-F},
	\begin{align*}
	0=\nabla_\mu F=&\dot{F}^{\mu\mu}\nabla_\mu h_{\mu\mu}+\sum_{\a=1}^{n-1}\dot{F}^{\a\a}\nabla_\mu h_{\a\a}.
	\end{align*}
If $\Sigma_t$ is convex, then $\dot{F}^{\mu\mu}$ is positive and thus we have
\begin{align*}
\nabla_\mu H=&\sum_{\a=1}^{n-1}\nabla_\mu h_{\a\a}+\nabla_\mu h_{\mu\mu}\\
	=&-\sum_{\a=1}^{n-1}\frac{\dot{F}^{\a\a}}{\dot{F}^{\mu\mu}}\nabla_\mu h_{\a\a}+\sum_{\a=1}^{n-1} \nabla_\mu h_{\a\a}\\
	=&\sum_{\a=1}^{n-1} \frac{1}{\dot{F}^{\mu\mu}}(\dot{F}^{\mu\mu}-\dot{F}^{\a\a})\nabla_\mu h_{\a\a}.
\end{align*}
By Proposition \ref{s2:prop-second-fundamental-form},
$$
	\nabla_{\mu}h_{\a\a}=\cot\t {h}_{\a\a}(h_{\mu\mu}-h_{\a\a}), \quad \a=1,\cdots,n-1,
	$$
which implies that
	\begin{align*}
	\nabla_\mu H=&\sum_{\a=1}^{n-1} \frac{1}{\dot{F}^{\mu\mu}}(\dot{F}^{\mu\mu}-\dot{F}^{\a\a})(h_{\mu\mu}-h_{\a\a}){h}_{\a\a}\cot\theta.
	\end{align*}
Since $F$ is concave, there holds $(\dot{F}^{\mu\mu}-\dot{F}^{\a\a})(h_{\mu\mu}-h_{\a\a})\leq 0$. By the convexity of $\Sigma_t$ and the condition $\theta\in (0,\frac{\pi}2]$, we also have
${h}_{\a\a}\cot\theta\geq 0$. Therefore, we conclude that $\nabla_\mu H\leq 0$ on the boundary $\partial\Sigma_t$.

(v). Using \eqref{s2:normal-general-f} and \eqref{s3.du}, we have
\begin{align*}
  \partial_tu= & \langle \partial_tx,\nu\rangle+\langle x,\partial_t \nu\rangle\\
  = & \frac{1+\cos\theta\langle \nu,e\rangle}{F}-u-\frac{\cos\theta}{F}\langle x,e_i\rangle h_i^k\langle e_k,e\rangle+\frac{1+\cos\theta\langle \nu,e\rangle}{F^2}\nabla F\\
  &+\langle x,\nabla u\rangle +\langle x,e_i\rangle h_i^k\langle e_k,V\rangle\\
  =&\left\langle V+x-\frac{\cos\theta}{F}e,\nabla u\right\rangle+\frac{1+\cos\theta\langle \nu,e\rangle}{F}-u+\frac{1+\cos\theta\langle \nu,e\rangle}{F^2}\nabla F.
\end{align*}
On the other hand, the equation \eqref{s3.d2u} implies that
\begin{equation*}
  \dot{F}^{k\ell}\nabla_k\nabla_\ell u=F+\langle x,\nabla F\rangle -\dot{F}^{ij}(h^2)_{ij}u.
\end{equation*}
Combining the above two equations, we obtain \eqref{s3.evl-u}.

(vi). Similarly, by \eqref{s2:normal-general-f} and \eqref{s3.dnue}, we have
\begin{align*}
  \partial_t\langle \nu,e\rangle= & \langle \partial_t \nu,e\rangle \\
  = &-\frac{\cos\theta}{F}\langle h_i^ke_k,e\rangle\langle e_i,e\rangle + \frac{1+\cos\theta\langle \nu,e\rangle}{F^2}\langle\nabla F,e\rangle\\
  &+\nabla_i\langle x,\nu\rangle \langle e_i,e\rangle+\langle V,e_k\rangle h_i^k\langle e_i,e\rangle \\
  =&\left\langle V+x-\frac{\cos\theta}{F}e,\nabla \langle \nu,e\rangle\right\rangle+ \frac{1+\cos\theta\langle \nu,e\rangle}{F^2}\langle\nabla F,e\rangle.
\end{align*}
The equation \eqref{s3.d2nue} implies that
\begin{equation*}
  \dot{F}^{k\ell}\nabla_k\nabla_\ell \langle \nu,e\rangle=\langle \nabla F,e\rangle -\dot{F}^{ij}(h^2)_{ij}\langle \nu,e\rangle.
\end{equation*}
Combining the above two equations, we obtain \eqref{s3.evl-nue}.
\end{proof}

\section{Tensor maximum principle}\label{sec:4}
In this section, we prove the tensor maximum principle in Theorem \ref{s1:thm-max principle}.

Let $\Sigma$ be a smooth compact manifold with boundary $\partial \Sigma$ and $\mu$ be the outward pointing unit normal vector field of $\partial\Sigma$ in $\Sigma$.
Denote $T\Sigma$ the tangent bundle of $\Sigma$, that is, $T\Sigma=\{(p,\xi),~p\in \Sigma, ~\xi\in T_p\Sigma\}$. Let $S=(S_{ij})$ be a smooth symmetric tensor on $\Sigma$. We consider a function on the tangent bundle of $\Sigma$:
\begin{equation}\label{s4.Z}
Z(p,\xi)=S(p)(\xi,\xi)
\end{equation}
where $(p,\xi)\in T\Sigma$. Then $S_{ij}\geq 0$ is equivalent to an inequality  $Z(p,\xi)\geq 0$ for all $(p,\xi)\in T\Sigma$.  Let $p\in \Sigma$ be a point where $S(p)$ has a null vector $\xi$, i.e., $S(p)(\xi,\xi)=0$.  The key of Andrew's argument \cite[\S 3]{And2007} is that the minimality of $Z(p,\xi)$ implies that the Hessian of $Z$ is non-negative definite, which is obvious when $p$ is an interior point of the manifold. We prove that this is also true on the boundary when $S_{ij}$ satisfies a proper Neumann boundary condition.
\begin{lem}\label{s3:tech-lemma}
Suppose that $S_{ij}\geq 0$ on $\Sigma$ and
\begin{equation}\label{s4.lem1}
  (\nabla_\mu S_{ij})\xi^i\xi^j\geq 0,\qquad\mathrm{ on}~ \partial\Sigma
\end{equation}
whenever $S_{ij}\xi^j=0$ for some tangent vector $\xi$. If the function $Z$ defined in \eqref{s4.Z} attains its minimum $Z(p,\xi)=0$ at $(p,\xi)\in T\Sigma$, then $D Z|_{(p,\xi)}=0$ and $D^2Z|_{(p,\xi)}$ is non-negative definite.
\end{lem}
\begin{proof}
We choose coordinates $x^1,\cdots,x^n$ for $\Sigma$ near $p$ such that the connection coefficients $\nabla$ vanish at $p$. Then any vector in $T_q{\Sigma}$ for $q$ near $p$ has the form $\sum_{i=1}^n{\dot{x}^i \partial_i}$, so $T\Sigma$ is described locally by the $2n$ coordinates $x^1,\cdots,x^n$ and $\dot{x}^1,\cdots,\dot{x}^n$. We can rotate the coordinates such that $\xi=\partial_1$. In particular, we have $Z(p,\xi)=0$ and $Z(q,\eta)\geq 0$ for all $q\in\Sigma$ and $\eta\in T_q\Sigma$.

If $p\in\operatorname{int}(\Sigma)$, then the conclusion follows easily. So in the following, we assume that $p\in\partial\Sigma$. For any $v\in T_p \partial\Sigma$, we have
   \begin{equation*}
	0=DZ|_{(p,\xi)}\cdot \left(v,0\right),
	\end{equation*}
	We need to show that $DZ|_{(p,\xi)}\cdot \left(\mu,0\right)=0$ as well, where $\mu\in T_p\Sigma$ is the unit outward normal of $\partial\Sigma$ in $\Sigma$. Since $Z(p,\xi)=S_{ij}\xi^i\xi^j=0$,  by the assumption \eqref{s4.lem1}, we have
	\begin{equation*}
	0\leq \left(\nabla_\mu S_{ij}\right)\xi^i\xi^j=DZ|_{(p,\xi)}\cdot \left(\mu,0\right).
	\end{equation*}
If $DZ|_{(p,\xi)}\cdot \left(\mu,0\right)>0$, we consider the geodesic $\gamma:[0,\delta]\rightarrow \Sigma$ with $\gamma(0)=p$ and $\gamma'(0)=-\mu$. Extend $\xi\in T_p\Sigma$ to a vector field $\xi(s)\in T_{\gamma(s)}\Sigma$ along the geodesic $\gamma$ by parallel transport, that is, $\xi'(s)=\nabla_{\gamma'(s)}\xi(s)=0$ and $\xi(0)=\xi$.  Along the geodesic $\gamma(s)$, we have
\begin{align*}
  \frac{d}{ds}\Big |_{s=0}Z(\gamma(s),\xi(s))=&DZ|_{(p,\xi)}\cdot (\gamma'(0),\xi'(0))\\
  =&-DZ|_{(p,\xi)}\cdot \left(\mu,0\right)<0
\end{align*}
 and hence there exists an interior point $\gamma(s_0)\in\text{int}(\Sigma), 0<s_0<\delta$ such that $Z(\gamma(s_0),\xi(s_0))<Z(\gamma(0),\xi(0))=Z(p,\xi)=0$, which contradicts with $Z\geq 0$. Therefore,
 \begin{equation*}
   DZ|_{(p,\xi)}\cdot (v,0)=0,\quad \forall~v\in T_p\Sigma.
 \end{equation*}
 Since $\xi\in T_p\Sigma$ is an interior point, we also have
  \begin{equation*}
   DZ|_{(p,\xi)}\cdot (0,\eta)=0,\quad \forall~\eta\in T_\xi(T_p\Sigma).
 \end{equation*}
 This means that $DZ|_{(p,\xi)}=0$. In local coordinates, we write this as
	\begin{equation*}
	0=\frac{\partial Z}{\partial x^k}=\frac{\partial}{\partial x^k}S_{11}
	\end{equation*}
and
	\begin{equation*}
	0=\frac{\partial Z}{\partial \dot{x}^k}=2S_{k1}
	\end{equation*}
	for all $k=1,\cdots,n$.

Next, we show that $D^2Z$ is non-negative definite, i.e., for any $a=(a_1,\cdots,a_n)$, $b=(b_1,\cdots,b_n)\in\mathbb{R}^n$, at the point $(p,\xi)$ we have
	\begin{equation*}
	D^2Z|_{(p,\xi)}\left((a,b),(a,b)\right)\geq 0.
	\end{equation*}
	Suppose not, then there exists $(a,b)$ such that
	\begin{equation*}
	D^2Z|_{(p,\xi)}\left((a,b),(a,b)\right)<0.
	\end{equation*}
	Note that for any fixed two vectors $\tilde{a},\tilde{b}\in\mathbb{R}^{n}$, we have
	\begin{align*}
	(D^2Z)|_{(p,\xi)}((0,\xi),(\tilde{a},\tilde{b}))=&2\nabla_{\tilde{a}}S_{11}+2\sum_{k=1}^{n}S_{1k}\tilde{b}_k\\
=&DZ|_{(p,\xi)}\cdot \left(2\tilde{a},\tilde{b}\right)=0.
	\end{align*}
	Hence there exits a constant $k\in\mathbb{R}$ such that $b-k\xi\in T_{\xi}\mathbb{S}^{n-1}$ and
	\begin{equation*}
	D^2Z|_{(p,\xi)}\left((a,b-k\xi),(a,b-k\xi))\right)=D^2Z|_{(p,\xi)}\left((a,b),(a,b)\right)<0
	\end{equation*}
	at the point $(p,\xi)$. Without loss of generality, we assume that $\langle a,\mu\rangle<0$, i.e., $a\in T_p\Sigma$ is pointing inward of $\Sigma$. Denote $UT(\Sigma)$ the unit tangent bundle of $\Sigma$, that is, $UT(\Sigma)=\{(p,\xi)\in T\Sigma, ~|\xi|=1\}$. Let $\tilde{\gamma}$ be a curve in $UT(\Sigma)$  such that $\tilde{\gamma}(s)=(p(s),\xi(s))$ with $\tilde{\gamma}(0)=(p,\xi),\ \tilde{\gamma}'(0)=(a,b-k\xi)$. Then we have
	\begin{equation*}
	\frac{d}{ds}\Big|_{s=0}Z(\tilde{\gamma}(s))=DZ|_{(p,\xi)}\cdot(a,b-k\xi)=0,
	\end{equation*}
	and
	\begin{align*} \frac{d^2}{ds^2}\Big|_{s=0}Z(\tilde{\gamma}(s))=&D^2Z|_{(p,\xi)}(\tilde{\gamma}'(0),\tilde{\gamma}'(0))+DZ|_{(p,\xi)}\cdot\tilde{\gamma}''(0)\\
	=&D^2Z|_{(p,\xi)}\left((a,b-k\xi),(a,b-k\xi))\right)<0,
	\end{align*}
	which contradicts with the minimality of $Z(\tilde{\gamma}(s))$ at $s=0$.
\end{proof}

By Lemma \ref{s3:tech-lemma}, we can use the argument in \cite[\S 3]{And2007} to prove Theorem \ref{s1:thm-max principle}.
\begin{proof}[Proof of Theorem \ref{s1:thm-max principle}]
Recall that the coordinate is chosen such that $\partial_1=\xi$. Consider the second order conditions implied by the minimality of $Z$, the entries in the Hessian matrix $D^2Z$ are as follows:
	\begin{align*}
	\frac{\partial^2 Z}{\partial x^k \partial x^\ell}&=\frac{\partial^2}{\partial x^k \partial x^\ell} S_{11},\\
	\frac{\partial^2 Z}{\partial x^k \partial \dot{x}^q}&=2\frac{\partial}{\partial x^k} S_{1q},\\
	\frac{\partial^2 Z}{\partial \dot{x}^p \partial \dot{x}^q}&=2 S_{pq},
	\end{align*}
where $1\leq k,\ell,p,q\leq n$.	Note that in the coordinates chosen above, $\nabla_k\nabla_\ell S_{11}=\frac{\partial^2S_{11}}{\partial x^k \partial x^\ell}$ at $p$. Since $S_{ij}$ satisfies the boundary condition \eqref{eq-bou} whenever $S_{ij}\geq 0$ and $S_{ij}\xi^j=0$, by Lemma \ref{s3:tech-lemma} we have $D^2Z\geq 0$ at $(p,\xi)$. Then for any matrix $\Gamma=(\Gamma_k^p)_{n\times n}$, we have:
	\begin{align*}
	0&\leq a^{k\ell}\(\frac{\partial^2 Z}{\partial x^k \partial x^\ell}-\Gamma_k^p\frac{\partial^2 Z}{\partial x^\ell \partial \dot{x}^p}-\Gamma_\ell^q \frac{\partial^2 Z}{\partial x^k \partial \dot{x}^q}+\Gamma_k^p\Gamma_\ell^q\frac{\partial^2 Z}{\partial \dot{x}^p \partial \dot{x}^q}\)\\
	&=a^{k\ell}\(\frac{\partial^2 S_{11}}{\partial x^k \partial x^\ell}-4\Gamma_k^p\frac{\partial S_{1p}}{\partial x^\ell}+2\Gamma_k^p\Gamma_\ell^q S_{pq}\).
	\end{align*}
	It follows that at the minimum point $(p,\xi)$, the function $Z$ satisfies
	\begin{align*}
	\frac{\partial Z}{\partial t}=&a^{k\ell}\nabla_k\nabla_\ell S_{ij}\dot{x}^i\dot{x}^j+b^k\nabla_kS_{ij}\dot{x}^i\dot{x}^j+N_{ij}\dot{x}^i\dot{x}^j\\
	=&a^{k\ell}\frac{\partial^2S_{11}}{\partial x^k \partial x^\ell}+N_{11}\\
	\geq&\sup_{\Gamma}2a^{k\ell}\left(2\Gamma_k^p\nabla_\ell S_{1p}-\Gamma_k^p\Gamma_\ell^qS_{pq}\right)+N_{11}\\
	\geq&0,
	\end{align*}
by the condition \eqref{conditon-MP}.	By the maximum principle (see the argument in Lemma 3.1 and Lemma 3.5 of \cite{Hamilton}), we conclude that $Z\geq 0$ is preserved.
\end{proof}

\section{Preserving of convexity}\label{sec:5}

Since the initial hypersurface $\Sigma_0$ is strictly convex, we can choose a suitable point $o$ on $\widehat{\partial\Sigma}$ as the origin of $\mathbb{R}^{n+1}$ such that $\Sigma_0$ is star-shaped with respect to $o$. Moreover, there exists two constants $0<r_1<r_2<\infty$ such that
\begin{equation*}
  \Sigma_0 \subset \widehat{C_{r_2,\t}} \backslash \widehat{C_{r_1,\t}},
\end{equation*}
where $\widehat{C_{r_i,\t}}, i=1,2$ are spherical caps of radius $r_i$. As the spherical caps are stationary along the flow \eqref{s1:BGL-flow}, the avoidance principle implies that the solution $\Sigma_t$ of the flow \eqref{s2:BGL-flow} satisfies
\begin{align}\label{s4:barrier-estimate}
	\Sigma_t \subset \widehat{C_{r_2,\t}} \backslash \widehat{C_{r_1,\t}},
	\end{align}
for all time  $t\in [0,T)$. This gives the $C^0$ estimate of $\Sigma_t$.

The star-shapedness of $\Sigma_t$ is preserved: this follows from applying maximum principle to a modified support function introduced in \cite[\S 2]{WWX23}:
\begin{equation*}
  \bar{u}=\frac{u}{1+\cos\theta\langle \nu,e\rangle}.
\end{equation*}
Combining \eqref{s3.evl-u} and \eqref{s3.evl-nue}, $\bar{u}$ satisfies the evolution equation
\begin{align}\label{s5.evl-ubar}
  \partial_t\bar{u}= & \frac{1+\cos\t\langle \nu,e\rangle}{F^2}\dot{F}^{k\ell}\nabla_k \nabla_\ell \bar{u} +\left\langle V+x-\frac{\cos\t}{F}e,\nabla \bar{u}\right\rangle\nonumber\\
  & +\bar{u}\left(\frac{\dot{F}^{ij}(h^2)_{ij}}{F^2}-1\right)+\frac{2\cos\theta}{F^2}\dot{F}^{ij}\nabla_i\langle\nu,e\rangle\nabla_j\bar{u}.
\end{align}
On the boundary $\partial\Sigma_t$, using \eqref{s2:relation-normal} we have
\begin{align*}
  \nabla_\mu u= & \langle \mu,\nu\rangle +\langle x, h(\mu,\mu)\mu\rangle  \\
  =& h(\mu,\mu) \cos\theta\langle x, \bar{\nu}\rangle =  h(\mu,\mu)\cot\theta u,\\
  \nabla_\mu\langle\nu,e\rangle=&\langle h(\mu,\mu)\mu,e\rangle \\
  =&h(\mu,\mu)\langle \sin\theta e,e\rangle=-h(\mu,\mu)\tan\theta\langle \nu,e\rangle.
\end{align*}
These imply that
\begin{align}\label{s5.dubar}
  \nabla_\mu\bar{u} =&  \frac{uh(\mu,\mu)}{(1+\cos\t\langle \nu,e\rangle)^2}\left(\cot\theta+\frac{\langle \nu,e\rangle}{\sin\theta}\right) =0,\qquad \text{on}\quad \partial\Sigma_t.
\end{align}
Applying maximum principle to \eqref{s5.evl-ubar} and using the boundary condition \eqref{s5.dubar}, we have that $\bar{u}\geq \min_{\Sigma_0}\bar{u}>0$ on $\Sigma_t$ and then $u\geq (1-\cos\theta)\bar{u}>0$. That is, $\Sigma_t$ is star-shaped.

By considering  $\bar{F}=\bar{u}F$, a positive lower bound on $F$ is proved in \cite[Prop. 2.6]{WWX23}. We include the proof here  for the convenience of readers.
\begin{prop}\label{s5.Flbd}
Assume that the initial hypersurface $\Sigma_0$ is star-shaped and satisfies $F>0$. Let $\Sigma_t, t\in [0,T)$, be the smooth solution to the flow \eqref{s2:BGL-flow} starting from $\Sigma_0$. Then there exists a positive constant $C$ such that $F\geq C>0$ on $\Sigma_t$ for all $t\in[0,T)$.
\end{prop}
\proof
As in \cite[Prop. 2.6]{WWX23}, we consider $\bar{F}=\bar{u}F$. The equations \eqref{s3:evol-F} and \eqref{s5.evl-ubar} imply that
\begin{align}\label{s5.evl-Fbar}
 \partial_t\bar{F}= & \frac{1+\cos\t\langle \nu,e\rangle}{F^2}\dot{F}^{k\ell}\nabla_k \nabla_\ell \bar{F} +\left\langle V+x-\frac{\cos\t}{F}e,\nabla \bar{F}\right\rangle\nonumber\\
   &+\frac{2\cos\theta}{F^2}\dot{F}^{ij}\nabla_i \langle \nu,e\rangle\nabla_j\bar{F}-\frac{2(1+\cos\t\langle \nu,e\rangle)}{F^3}\dot{F}^{ij}\nabla_iF\nabla_j\bar{F}.
\end{align}
On the boundary $\partial\Sigma_t$, by \eqref{s3:bdry-F} and \eqref{s5.dubar} we also have
\begin{align}\label{s5.dFbar}
  \nabla_\mu\bar{F} =&  0,\qquad \text{on}\quad \partial\Sigma_t.
\end{align}
The maximum principle implies that
\begin{equation*}
  \bar{u}F=\bar{F}\geq \min_{\Sigma_0}\bar{u}F>0.
\end{equation*}
Since $\bar{u}$ is bounded from above uniformly:
\begin{equation*}
  \bar{u}\leq \frac{|x|}{1-\cos\theta},
\end{equation*}
we conclude that $F\geq C>0$ on $\Sigma_t$ for a uniform positive constant $C$.
\endproof

Now we apply the tensor maximum principle in Theorem \ref{s1:thm-max principle} to prove that the flow \eqref{s2:BGL-flow} preserves the strict convexity of $\Sigma_t$.
\begin{thm}\label{s4:thm-preserving-strict-convex}
Assume that the initial hypersurface $\Sigma_0$ is strictly convex. Let $\Sigma_t, t\in [0,T)$, be the smooth solution to the flow \eqref{s2:BGL-flow} starting from $\Sigma_0$. Then there exists a constant $\varepsilon=\varepsilon(\Sigma_0)$ such that $h_{ij}\geq \varepsilon Fg_{ij}$ is preserved for positive times $t\in[0,T)$.
\end{thm}
\begin{proof}
Since the initial hypersurface $\Sigma_0$ is strictly convex, we have $F>0$ on $\Sigma_0$ and by the compactness there exists a positive constant $\varepsilon=\varepsilon(\Sigma_0)$ such that $h_{ij}\geq \varepsilon Fg_{ij}$ holds on $\Sigma_0$.	Let
\begin{equation*}
  S_{ij}=h_{ij}-\varepsilon Fg_{ij}.
\end{equation*}
We will apply the tensor maximum principle in Theorem \ref{s1:thm-max principle} to show that $S_{ij}\geq 0$ is preserved. This together with Proposition \ref{s5.Flbd} concludes the theorem.
	
 Combining the evolution equations \eqref{s3:evol-metric} -- \eqref{s3:evol-F}, we calculate that the tensor $S_{ij}=h_{ij}-\e Fg_{ij}$ evolves by
		\begin{align}\label{s3:evol-Sij}
		\partial_tS_{ij}=&\partial_th_{ij}-\varepsilon \partial_tFg_{ij}-\varepsilon F\partial_tg_{ij}\nonumber\\
=&\frac{1+\cos\t\langle \nu,e\rangle}{F^2}\dot{F}^{k\ell}\nabla_k \nabla_\ell S_{ij}+\left\langle V+x-\frac{\cos\t}{F}e,\nabla S_{ij}\right\rangle\nonumber\\
		&+\frac{1+\cos\t\langle \nu,e\rangle}{F^2}\ddot{F}^{k\ell,pq}\nabla_i h_{k\ell}\nabla_j h_{pq} -\frac{2(1+\cos\t\langle \nu,e\rangle)}{F^3}\nabla_iF\nabla_jF\nonumber\\
		&+\frac{\cos\t}{F^2}(\nabla_i F h_j^k+\nabla_j F h_i^k)\langle e,e_k\rangle -\frac{2\varepsilon\cos\t}{F^2}\dot{F}^{k\ell}\nabla_k F h_\ell^j\langle e_j,e\rangle g_{ij}\nonumber\\
&+\frac{2\varepsilon\left(1+\cos\theta\langle\nu,e\rangle\right)}{F^3}\dot{F}^{k\ell}\nabla_kF\nabla_\ell F g_{ij}\nonumber\\
		&+\(\frac{\cos\t\langle \nu,e\rangle}{F}-2\langle x,\nu\rangle\) ((S^2)_{ij}+2\e FS_{ij})+S_j^k \nabla_i V_k+S_i^k\nabla_j V_k \nonumber\\
		&+\(1+\frac{1+\cos\t\langle \nu,e\rangle}{F^2}\dot{F}^{k\ell}(h^2)_{k\ell}-2\e(1+\cos\t\langle \nu,e\rangle)+2\e\langle x,\nu\rangle F\)S_{ij}\nonumber\\
		&+\frac{\e(2+\cos\t\langle \nu,e\rangle)}{F}\left(\dot{F}^{k\ell}(h^2)_{k\ell}-\e F^2\right)g_{ij}.
		\end{align}
We need to check both the conditions \eqref{conditon-MP} and \eqref{eq-bou} in order to apply Theorem \ref{s1:thm-max principle}.

We first check the boundary condition \eqref{eq-bou}. Assume that $S_{ij}\geq 0$, and $S_{ij}\xi^j=0$ holds at a boundary point $p_0\in \partial\Sigma_{t_0}$, where $\xi$ is a null vector.    Let $\{e_\a\}_{\a=1}^{n-1}$ be a local orthonormal frame of $\partial \Sigma_{t_0}\subset \Sigma_{t_0}$, then $\{\mu\}\cup\{e_\a\}_{\a=1}^{n-1}$ forms an orthonormal frame of $\Sigma_{t_0}$ and we have either $\xi=\mu$ or $\xi=e_\a$ for some $\a$ on $\partial \Sigma_{t_0}$. By Proposition \ref{s2:prop-second-fundamental-form} and $\nabla_\mu F=0$ on $\partial\Sigma_{t_0}$, we have:
	\begin{enumerate}[(i)]
		\item If $\xi=e_\a$, then $h_{\mu\mu}\geq h_{\a\a}$ and
		\begin{align*}
		(\nabla_\mu S_{ij})\xi^i\xi^j=&\nabla_\mu h_{\a\a}-\varepsilon\nabla_\mu F\\
=&\tilde{h}_{\a\a}(h_{\mu\mu}-h_{\a\a})=\cot\t h_{\a\a}(h_{\mu\mu}-h_{\a\a})\geq 0.
		\end{align*}
		\item If $\xi=\mu$, then $h_{\a\a}\geq h_{\mu\mu}$ for all $\a=1,\cdots,n-1$. By \eqref{s3:bdry-F}, we have
		\begin{align*}
		0=\nabla_\mu F=\dot{F}^{\mu\mu}\nabla_{\mu} h_{\mu\mu}+\sum_{\a=1}^{n-1}\dot{F}^{\a\a}\nabla_{\mu} h_{\a\a}.
		\end{align*}
		Then we get
		\begin{align*}
		(\nabla_\mu S_{ij})\xi^i\xi^j=&\nabla_\mu h_{\mu\mu}-\varepsilon\nabla_\mu F=-\sum_{\a=1}^{n-1}\frac{\dot{F}^{\a\a}}{\dot{F}^{\mu\mu}}\nabla_{\mu} h_{\a\a}\\
		=&\sum_{\a=1}^{n-1}\frac{\dot{F}^{\a\a}}{\dot{F}^{\mu\mu}}\cot\t h_{\a\a}\(h_{\a\a}-h_{\mu\mu}\)\geq 0.
		\end{align*}
	\end{enumerate}
Therefore, the boundary condition \eqref{eq-bou} is satisfied.

We next check the condition \eqref{conditon-MP}. Assume that $S_{ij}\geq 0$, and $S_{ij}\xi^j=0$ holds at a point $(p_0,t_0)$, where $\xi$ is a null vector. The terms in \eqref{s3:evol-Sij} involving $S_{ij}$ and $(S^2)_{ij}$ can be ignored due to the null vector condition.  We choose an orthonormal frame $\{e_1,\cdots, e_n\}$ around $p_0$ such that $(h_{ij})$ is diagonal with principal curvatures $\kappa=(\kappa_1,\cdots,\kappa_n)$ in the increasing order and $e_i$ corresponds to the principal direction of $\kappa_i$. The the null vector $\xi=e_1$ is the eigenvector corresponding to $\kappa_1$. We need to show that
    \begin{align}\label{s4:key-ineq-1}
    Q_1:=&\frac{1+\cos\t\langle \nu,e\rangle}{F^2}\(\ddot{F}^{k\ell,pq}\nabla_1 h_{k\ell}\nabla_1 h_{pq}-\frac{2}{F}|\nabla_1F|^2\)\nonumber\\
    &+\frac{2\varepsilon\left(1+\cos\theta\langle\nu,e\rangle\right)}{F^3}\dot{F}^{k\ell}\nabla_kF\nabla_\ell F\nonumber\\
    &+\frac{2\varepsilon\cos\t}{F}\nabla_1 F \langle e,e_1\rangle-\frac{2\varepsilon\cos\t}{F^2}\dot{F}^{k\ell}\nabla_k F h_\ell^j\langle e_j,e\rangle\nonumber\\
    &+\frac{\e(2+\cos\t\langle \nu,e\rangle)}{F}\left(\dot{F}^{k\ell}(h^2)_{k\ell}-\e F^2\right)\nonumber\\
    &+\frac{2(1+\cos\t\langle \nu,e\rangle)}{F^2}\sup_{\G}\dot{F}^{k\ell}(2\G_k^p \nabla_\ell S_{1p}-\G_k^p\G_\ell^qS_{pq})\geq 0.
    \end{align}

By continuity we can assume that the principal curvatures are mutually distinct at $(x_0,t_0)$ and satisfy $\kappa_1<\kappa_2<\cdots<\kappa_n$. This is reasonable (see \cite[p.30]{And2007}) since for any positive definite symmetric matrix $W$ with $W_{ij}\geq \varepsilon F(W)\delta_{ij}$ and $W_{ij}\xi^i\xi^j=\varepsilon F(W)|\xi|^2$ for some $\xi\neq 0$, there is a sequence of symmetric matrixes $\{W^{(k)}\}$ approaching $W$, satisfying $W^{(k)}_{ij}\geq \varepsilon F(W^{(k)})\delta_{ij}$ and $W^{(k)}_{ij}\xi^i\xi^j=\varepsilon F(W^{(k)})|\xi|^2$ and with each $W^{(k)}$ having distinct eigenvalues. Hence it suffice to prove the result in the case where all of $\kappa_i$ are distinct.

Since $S_{11}=\k_1-\e F=0$ at $(x_0,t_0)$ and $S_{\ell\ell}>0$ for $\ell=2,\cdots,n$. By Lemma \ref{s3:tech-lemma}, the minimality of $S_{11}$ implies that
\begin{equation}\label{s5.DS}
0=\nabla_k S_{11}=\nabla_kh_{11}-\varepsilon\nabla_kF
\end{equation}
 for all $k=1,\cdots,n$.   Using the estimate \eqref{s2:F-ddt} and \eqref{s2.d2F}, we have
    \begin{align}\label{s5.Q1-0}
    &\ddot{F}^{k\ell,pq}\nabla_1 h_{k\ell}\nabla_1 h_{pq}-\frac{2}{F}|\nabla_1 F|^2  \nonumber\\
   =&\ddot{f}^{k\ell}\nabla_1h_{kk}\nabla_1h_{\ell\ell}+2\sum_{k>\ell}\frac{\dot{f}^k-\dot{f}^\ell}{\kappa_k-\kappa_\ell}(\nabla_1h_{k\ell})^2-\frac{2}{F}|\nabla_1 F|^2\nonumber\\
   \geq & -2 \sum_{k=1}^n\frac{\dot{f}^k}{\k_k}|\nabla_1 h_{kk}|^2+2\sum_{k>\ell}\frac{\dot{f}^k-\dot{f}^\ell}{\kappa_k-\kappa_\ell}(\nabla_1h_{k\ell})^2.
    \end{align}
Using \eqref{s5.DS}, we rewrite the second line of \eqref{s4:key-ineq-1} as follows
    \begin{align}
    \text{Line 2 of } \eqref{s4:key-ineq-1} =& \frac{2\varepsilon\left(1+\cos\theta\langle\nu,e\rangle\right)}{F^3}\sum_{k=1}^{n}\dot{f}^k\frac{(\nabla_kh_{11})^2}{\varepsilon^2}\nonumber\\
      =&  \frac{2\left(1+\cos\theta\langle\nu,e\rangle\right)}{F^2}\sum_{k=1}^{n}\frac{\dot{f}^k}{\kappa_1}(\nabla_kh_{11})^2. \label{s5.Q1-1}
    \end{align}
  Similarly,
    \begin{align}
      \text{Line 3 of } \eqref{s4:key-ineq-1} = & \frac{2\varepsilon\cos\t}{F}\sum_{k=1}^{n}\dot{f}^k\nabla_1h_{kk} \langle e,e_1\rangle-\frac{2\varepsilon\cos\t}{F^2}\sum_{k=1}^{n}\dot{f}^{k}\frac{\nabla_kh_{11}}{\varepsilon} \kappa_k\langle e_k,e\rangle \nonumber\\
      = & \frac{2\varepsilon\cos\t}{F}\sum_{k>1}\dot{f}^k\nabla_1h_{kk} \langle e,e_1\rangle-\frac{2\cos\t}{F^2}\sum_{k>1}\dot{f}^{k}\kappa_k\nabla_kh_{11}\langle e_k,e\rangle. \label{s5.Q1-2}
    \end{align}
    Since $\kappa_1=\varepsilon F$ and $\kappa_\ell\geq \varepsilon F$ for $\ell=2,\cdots,n$,  for the fourth line of \eqref{s4:key-ineq-1} we have
    \begin{align}\label{s5.Q1-3}
      &\dot{F}^{k\ell}(h^2)_{k\ell}-   \varepsilon  F^2= \sum_{\ell=2}^{n}\dot{f}^\ell\kappa_\ell(\kappa_\ell-\kappa_1)\geq 0.
    \end{align}
    Finally, the last line in \eqref{s4:key-ineq-1} can be explicitly computed as follows
    \begin{align}\label{s5.Q1-4}
    &\sup_\Gamma\dot{F}^{k\ell}(2\G_k^p\nabla_\ell S_{1p}-\G_k^p\G_\ell^q S_{pq})\nonumber\\
    =&\sup_\Gamma\sum_{k=1}^n\dot{f}^{k}\sum_{p=2}^{n}\left(2\G^p_k\nabla_k S_{1p}-(\G_k^p)^2S_{pp}\right)\nonumber\\
        =&\sup_\Gamma\sum_{k=1}^n\dot{f}^{k}\sum_{p=2}^{n}\left(2\G^p_k\nabla_k h_{1p}-(\G_k^p)^2S_{pp}\right)\nonumber\\
      =&\sup_\Gamma\sum_{k=1}^n\dot{f}^{k}\sum_{\ell=2}^{n}\(\frac{|\nabla_1 h_{k\ell}|^2}{S_{\ell\ell}}-\(\G_k^\ell-\frac{\nabla_1 h_{k\ell}}{S_{\ell\ell}}\)^2S_{\ell\ell}\)\nonumber\\
   = &\sum_{k=1}^n\dot{f}^{k}\sum_{\ell=2}^{n}\frac{|\nabla_1 h_{k\ell}|^2}{S_{\ell\ell}}\nonumber\\
   =&\sum_{k=1}^n\sum_{\ell=2}^{n}\frac{\dot{f}^{k}}{\kappa_{\ell}-\kappa_1}|\nabla_1 h_{k\ell}|^2,
    \end{align}
    where in the fourth equality we have chosen $\G_k^\ell=\frac{\nabla_1 h_{k\ell}}{S_{\ell\ell}}$ for $\ell=2,\cdots,n$.  Combining \eqref{s5.Q1-0} -- \eqref{s5.Q1-4},  we have the following estimate for $Q_1$:
    \begin{align}\label{s5.Q2}
     Q_1\geq &\frac{1+\cos\t\langle \nu,e\rangle}{F^2}\biggl( -2 \sum_{k\geq 1}\frac{\dot{f}^k}{\k_k}|\nabla_1 h_{kk}|^2+2\sum_{k>\ell}\frac{\dot{f}^k-\dot{f}^\ell}{\kappa_k-\kappa_\ell}(\nabla_1h_{k\ell})^2\nonumber\\
     &+2\sum_{k\geq 1}\frac{\dot{f}^k}{\kappa_1}(\nabla_kh_{11})^2+2\sum_{k\geq 1,\ell>1}\frac{\dot{f}^{k}}{\kappa_{\ell}-\kappa_1}|\nabla_1 h_{k\ell}|^2\biggr)\nonumber\\
    &+\frac{2\varepsilon\cos\t}{F}\sum_{k>1}\dot{f}^k\nabla_1h_{kk} \langle e,e_1\rangle-\frac{2\cos\t}{F^2}\sum_{k>1}\dot{f}^{k}\kappa_k\nabla_kh_{11}\langle e_k,e\rangle \nonumber\\
    &+\frac{\e(2+\cos\t\langle \nu,e\rangle)}{F}\sum_{k>1}\dot{f}^k\kappa_k(\kappa_k-\kappa_1).
    \end{align}

 To show that $Q_1\geq 0$, we  estimate the first two lines of \eqref{s5.Q2} as follows:
  \begin{align*}
   \text{ Lines 1-2 of }\eqref{s5.Q2}= & \frac{1+\cos\t\langle \nu,e\rangle}{F^2}\biggl( -2\frac{\dot{f}^1}{\kappa_1}(\nabla_1h_{11})^2-2 \sum_{k>1}\frac{\dot{f}^k}{\k_k}|\nabla_1 h_{kk}|^2\nonumber\\
   &+2\sum_{k>\ell>1}\frac{\dot{f}^k-\dot{f}^\ell}{\kappa_k-\kappa_\ell}(\nabla_1h_{k\ell})^2+2\sum_{k>1}\frac{\dot{f}^k-\dot{f}^1}{\kappa_k-\kappa_1}(\nabla_kh_{11})^2\nonumber\\
     &+2\sum_{k\geq 1}\frac{\dot{f}^k}{\kappa_1}(\nabla_kh_{11})^2+2\sum_{k>1,\ell>1}\frac{\dot{f}^{k}}{\kappa_{\ell}-\kappa_1}|\nabla_1 h_{k\ell}|^2+2\sum_{\ell>1}\frac{\dot{f}^{1}}{\kappa_{\ell}-\kappa_1}|\nabla_\ell h_{11}|^2\biggr)\nonumber\\
   \geq & \frac{1+\cos\t\langle \nu,e\rangle}{F^2}\biggl( 2 \sum_{k>1}\left(\frac{\dot{f}^{k}}{\kappa_{k}-\kappa_1}-\frac{\dot{f}^k}{\k_k}\right)|\nabla_1 h_{kk}|^2\nonumber\\
   &+2\sum_{k>1}\left(\frac{\dot{f}^k}{\kappa_1}+\frac{\dot{f}^k}{\kappa_k-\kappa_1}\right)(\nabla_kh_{11})^2\nonumber\\
     &+2\sum_{k>\ell>1}\left(\frac{\dot{f}^k-\dot{f}^\ell}{\kappa_k-\kappa_\ell}+\frac{\dot{f}^{k}}{\kappa_{\ell}-\kappa_1}+\frac{\dot{f}^{\ell}}{\kappa_{k}-\kappa_1}\right)(\nabla_1h_{k\ell})^2\biggr)\nonumber\\
      \geq & \frac{1+\cos\t\langle \nu,e\rangle}{F^2}\biggl( 2 \sum_{k>1}\left(\frac{\dot{f}^{k}}{\kappa_{k}-\kappa_1}-\frac{\dot{f}^k}{\k_k}\right)|\nabla_1 h_{kk}|^2\nonumber\\
   &+2\sum_{k>1}\left(\frac{\dot{f}^k}{\kappa_1}+\frac{\dot{f}^k}{\kappa_k-\kappa_1}\right)(\nabla_kh_{11})^2\nonumber\\
     &+2\sum_{k>\ell>1}\left(-\frac{\dot{f}^k}{\kappa_\ell}-\frac{\dot{f}^\ell}{\kappa_k}+\frac{\dot{f}^{k}}{\kappa_{\ell}-\kappa_1}+\frac{\dot{f}^{\ell}}{\kappa_{k}-\kappa_1}\right)(\nabla_1h_{k\ell})^2\biggr)\nonumber\\
     \geq& \frac{2(1+\cos\t\langle \nu,e\rangle)}{F^2}\biggl( \sum_{k>1}\frac{\dot{f}^{k}\kappa_1}{\kappa_k(\kappa_{k}-\kappa_1)}|\nabla_1 h_{kk}|^2+\sum_{k>1}\frac{\dot{f}^k\kappa_k}{\kappa_1(\kappa_k-\kappa_1)}(\nabla_kh_{11})^2\biggr),
  \end{align*}
where in the second inequality we used \eqref{s2:inv-conc}.   This implies that
      \begin{align}\label{s5.Q2-1}
    Q_1\geq &\frac{2\varepsilon(1+\cos\t\langle \nu,e\rangle)}{F} \sum_{k>1}\frac{\dot{f}^{k}}{\kappa_k(\kappa_{k}-\kappa_1)}|\nabla_1 h_{kk}|^2+\frac{2\varepsilon\cos\t}{F}\sum_{k>1}\dot{f}^k\nabla_1h_{kk} \langle e,e_1\rangle\nonumber\\
   &+\frac{2(1+\cos\t\langle \nu,e\rangle)}{F^2} \sum_{k>1}\frac{\dot{f}^k\kappa_k}{\kappa_1(\kappa_k-\kappa_1)}(\nabla_kh_{11})^2-\frac{2\cos\t}{F^2}\sum_{k>1}\dot{f}^{k}\kappa_k\nabla_kh_{11}\langle e_k,e\rangle \nonumber\\
    &+\frac{\e(2+\cos\t\langle \nu,e\rangle)}{F}\sum_{k>1}\dot{f}^k\kappa_k(\kappa_k-\kappa_1).
    \end{align}
   Using the Young's inequality, the first line of \eqref{s5.Q2-1} can be estimated as
    \begin{align*}
     \text{Line 1 of }\eqref{s5.Q2-1} \geq  & -\frac{\varepsilon\cos^2\theta\langle e,e_1\rangle^2 }{2F(1+\cos\t\langle \nu,e\rangle)}\sum_{k>1}\dot{f}^k\kappa_k(\kappa_k-\kappa_1).
    \end{align*}
    Similarly,
     \begin{align*}
     \text{Line 2 of }\eqref{s5.Q2-1}\geq  & -\frac{\varepsilon\cos^2\theta}{2F(1+\cos\t\langle \nu,e\rangle)}\sum_{k>1}\dot{f}^k\kappa_k(\kappa_k-\kappa_1)\langle e_k,e\rangle^2\\
     \geq &-\frac{\varepsilon\cos^2\theta(1-\langle e,\nu\rangle^2-\langle e,e_1\rangle^2)}{2F(1+\cos\t\langle \nu,e\rangle)}\sum_{k>1}\dot{f}^k\kappa_k(\kappa_k-\kappa_1).
    \end{align*}
    Therefore, we conclude that
     \begin{align*}
    Q_1\geq &\biggl(\frac{\e(2+\cos\t\langle \nu,e\rangle)}{F}-\frac{\varepsilon\cos^2\theta(1-\langle e,\nu\rangle^2)}{2F(1+\cos\t\langle \nu,e\rangle)}\biggr)\sum_{k>1}\dot{f}^k\kappa_k(\kappa_k-\kappa_1)\nonumber\\
 =  & \frac{\varepsilon\sin^2\theta+3\varepsilon(1+\cos\theta\langle\nu,e\rangle)^2}{2F(1+\cos\t\langle \nu,e\rangle)}\sum_{k>1}\dot{f}^k\kappa_k(\kappa_k-\kappa_1)~\geq~0
    \end{align*}
at $(x_0,t_0)$.   This means that the condition \eqref{conditon-MP} is satisfied.

Applying Theorem \ref{s1:thm-max principle}, we conclude that $S_{ij}\geq 0$ on $M\times [0,T)$. This completes the proof of Theorem \ref{s4:thm-preserving-strict-convex}.
\end{proof}
\begin{rem}
The preservation of convexity can be proved for a large class of curvature function $F$. In fact, our proof also works for curvature function $F$ which is concave and inverse-concave in its arguments.
\end{rem}

\section{Proof of Theorem \ref{s1:thm-convergence} and Theorem \ref{s1:thm-AF-ineq}}\label{sec:6}

Since the initial hypersurface $\Sigma_0$ is strictly convex, by choosing a suitable point lying in $\widehat{\partial\Sigma_0}$ as the origin, $\Sigma_0$ is star-shaped with respect to the origin. We have proved in \S \ref{sec:5} that $\Sigma_t$, $t\in [0,T)$ is star-shaped as well. In the class of star-shaped hypersurfaces, we can reduce the flow \eqref{s1:BGL-flow} to a scalar parabolic equation with an oblique boundary condition. This also ensures the short time existence of the flow \eqref{s1:BGL-flow}.
\subsection{Radial graph parametrization}
Suppose the capillary hypersurface $\Sigma\subset \overline{\mathbb{R}}^{n+1}_+$ is star-shaped with respect to the origin, i.e., $\langle x,\nu\rangle>0$, one can reparametrize it as a graph over $\-{\mathbb S}^n_{+}$. That is, there exists a positive radial function $r$ defined on $\-{\mathbb S}^n_{+}$ such that
\begin{align*}
\Sigma=\{r(z)z~|~z\in\-{\mathbb S}^n_{+}\},
\end{align*}
where $z:=(z_1,\cdots,z_n)$ is a local coordinate of $\-{\mathbb S}^n_{+}$. We denote $\bar{\nabla}$ to be the Levi-Civita connection on $\mathbb S^n$ with respect to the round metric $\s:=g_{\mathbb S^n}$, $\partial_i=\partial_{z_i}$, $\s_{ij}=\s(\partial_i,\partial_j)$, $r_i=\bar{\nabla}_ir$ and $r_{ij}=\bar{\nabla}_i\bar{\nabla}_j r$. It is convenient to use a new radial function  $\varphi(z)=\log r(z)$ and denote
\begin{equation*}
  v=\sqrt{1+|\bar{\nabla}\varphi|^2}.
\end{equation*}
Then the induced metric $g$, the inverse matrix of $g$, the unit normal $\nu$, the second fundamental form $h$ and the Weingarten matrx $\mathcal{W}=(h_i^j)$ of $\Sigma$ can be expressed as follows (see \cite[\S 3]{guan2014}):
\begin{align*}
g_{ij}=&r^2 \d_{ij}+r_ir_j=e^{2\varphi}(\s_{ij}+\varphi_i\varphi_j),\\
g^{ij}=&e^{-2\varphi}\(\s^{ij}-\frac{\varphi^i\varphi^j}{v^2}\),\\
\nu=&\frac{1}{v}\(\partial_r-\frac 1{r^2}\bar{\nabla} r\)=\frac{1}{v}(\partial_r-\frac{1}{r}\bar{\nabla}\varphi),\\
h_{ij}=&\frac{e^\varphi}{v}(\s_{ij}+\varphi_i\varphi_j-\varphi_{ij}),\\
h_i^j=&g^{jk}h_{ki}=\frac{1}{e^\varphi v}\(\d_i^j-(\s^{jk}-\frac{\varphi^j\varphi^k}{v^2})\varphi_{ki} \).
\end{align*}
where $r^i=\s^{ij}r_j$, $\varphi^i=\s^{ij}\varphi_j$. The curvature function $F=\frac{H_\ell}{H_{\ell-1}}$ is a smooth symmetric function of the Weingarten matrix $h_i^j$ and thus can be expressed as a function of $\varphi,\varphi_i,\varphi_{ij}$. We can also express the support function of $\Sigma$ as
\begin{align*}
\langle x,\nu\rangle=\langle r\partial_r,\nu\rangle=\frac{e^\varphi}{v}.
\end{align*}
It's well known that the flow \eqref{s2:BGL-flow} is equivalent to the scalar parabolic equation for the radial function $r$ (and thus for $\varphi=\log r$) (see \cite{Ger06}):
\begin{equation}\label{s6.dr}
  \partial_t\varphi=\frac{1}{r}\partial_tr=\frac{v}{e^\varphi}\mathcal{F},\qquad \text{in}~ \mathbb S_+^n \times [0,T).
\end{equation}

The capillary boundary condition can also be expressed by the radial function $\varphi$ (see \cite[\S 4.1]{Wang-Weng-Xia2022}). We use the polar coordinate in the half-space. Let $x:=(x',x_{n+1})\in \mathbb R^n\times [0,\infty)$ and $z:=(\b,\xi)\in [0,\frac{\pi}{2}]\times \mathbb S^{n-1}$, we have
\begin{align*}
x_{n+1}=r\cos\b, \quad |x'|=r\sin\b.
\end{align*}
In this polar coordinates, the Euclidean metric is given by
$$
dx^2=dr^2+r^2(d\b^2+\sin^2\b g_{\mathbb S^{n-1}}).
$$
The vector $e_{n+1}=\cos\b \partial_r-\frac{\sin\b}{r}\partial_\b$.  Then we have
\begin{align*}
\langle \nu,e_{n+1}\rangle=\frac{1}{v}(\cos\b+\sin\b \bar{\nabla}_{\partial_\b}\varphi).
\end{align*}
On the boundary $\partial \mathbb S_{+}^n$, there holds
\begin{align*}
\-N \circ x=-e_{n+1}=\frac{1}{r}\partial_\b,
\end{align*}
and thus
\begin{align*}
-\cos\t =\langle \nu, \-N\circ x\rangle=-\frac{\bar{\nabla}_{\partial_\b} \varphi}{v},
\end{align*}
which we rewrite as
\begin{equation}\label{s6.dr2}
  \bar{\nabla}_{\partial_\beta}\varphi=\cos\theta \sqrt{1+|\bar{\nabla}\varphi|^2},\qquad \mathrm{on}~\partial\mathbb S_+^n \times [0,T).
\end{equation}
Since $0\leq \cos\t<1$, \eqref{s6.dr2} is  a uniform oblique boundary condition on $\partial\mathbb S_+^n$ as in \cite{Ural1991}. In fact, denote the boundary operator as
\begin{equation*}
  b(z,\varphi,p)=-\langle p,\partial_\beta\rangle +\cos\theta \sqrt{1+|p|^2},\quad p\in \mathbb{R}^n.
\end{equation*}
 Then \eqref{s6.dr2} is equivalent to $b(z,\varphi,\bar{\nabla}\varphi)=0$. As the unit inward normal of $\partial \mathbb S_{+}^n$ in $\mathbb S_{+}^n$ is $\gamma=-\partial_\beta$,  we have
 \begin{align*}
\sigma\left( \frac{\partial b}{\partial p}\Big|_{p=\bar{\nabla}\varphi}, \gamma\right)= &\sigma\left( \partial_\beta, \partial_\beta-\cos\theta \frac{p}{\sqrt{1+|p|^2}}\Big|_{p=\bar{\nabla}\varphi}\right)  \\
  =  & 1-\cos\theta \frac{\bar{\nabla}_{\partial_\beta}\varphi}{\sqrt{1+|\bar{\nabla}\varphi|^2}}\\
  =&1-\cos^2\theta>0.
 \end{align*}
 This means that oblique boundary condition \eqref{s6.dr2} is non-degeneracy as in \cite{Ural1991}.

In summary, under the radial graph parametrization, the flow \eqref{s1:BGL-flow} is transformed into a scalar parabolic equation with a uniform oblique boundary condition up to a tangential diffeomorphism:
\begin{align}\label{s4:scalar-pde}
\left\{\begin{aligned}
\partial_t\varphi=&\frac{v}{e^\varphi}\mathcal{F}, \quad &\text{in $\mathbb S_+^n \times [0,T)$,}\\
\bar{\nabla}_{\partial_\b}\varphi=&\cos\t \sqrt{1+|\bar{\nabla}\varphi|^2}, \quad &\text{on $\partial \mathbb S_{+}^n\times [0,T)$},\\
\varphi(\cdot,0)=&\varphi_0(\cdot), \quad &\text{on $\mathbb S^n_{+}$},
\end{aligned}\right.
\end{align}
where $\varphi_0$ is the radial function of $\Sigma_0$ over $\-{\mathbb S}_{+}^n$ and
$$
\mathcal{F}=\frac{H_{\ell-1}}{H_{\ell}}\left[ 1-\frac{\cos\t}{v}\(\cos\b+\sin \b\bar{\nabla}_{\partial_\b}\varphi\)\right]-\frac{e^\varphi}{v}.
$$

\subsection{Curvature estimates}
First, we deduce the uniform bound of $F$.
\begin{prop}\label{s4:prop-bound-F}
Along the flow \eqref{s2:BGL-flow}, if the initial hypersurface $\Sigma_0$ is strictly convex, then there holds
\begin{align}\label{s4:bounds-F}
0<C \leq F(p,t) \leq \max_{M} F(\cdot,0), \quad \forall (p,t)\in M\times [0,T),
\end{align}
where $\varepsilon>0$ is constant depending only on $\Sigma_0$.
\end{prop}
\begin{proof}
	The lower bound of $F$ has been proved in Proposition \ref{s5.Flbd}. For the upper bound of $F$, we note that $\dot{F}^{k\ell}(h^2)_{k\ell}\geq F^2$. By the evolution equation \eqref{s3:evol-F} of $F$ and the boundary condition $\nabla_\mu F=0$, the maximum principle (see \cite[Theorem 3.1]{Stahl1996-2}) implies that $F(p,t) \leq \max_{M} F(\cdot,0)$.
\end{proof}

Next, we prove the uniform upper bound of the mean curvature.
\begin{prop}\label{s4:upper-bound-H}
	There exists $C>0$ depending only on $\Sigma_0$ such that
	\begin{align*}
	H(p,t)\leq C, \quad \forall (p,t)\in [0, T).
	\end{align*}
\end{prop}
\begin{proof}
	The proof is by applying maximum principle to the evolution equation \eqref{s3:evol-mean-curvature} of $H$ and is essentially the same as \cite[Proposition 4.8]{Wang-Weng-Xia2022}. The only difference is that they used the fact $\dot{F}^{k\ell}(h^2)_{k\ell}=F^2$ due to that  $F=H_n/H_{n-1}$, while for general $F=H_\ell/H_{\ell-1}$ we instead use the inequality \eqref{s2:key-inequality}:
\begin{equation}\label{s6.F}
\dot{F}^{k\ell}(h^2)_{k\ell}\leq (n-\ell+1)F^2.
\end{equation}
 We include the proof here for convenience of the readers.

By the concavity of $F$, we have
	\begin{align*}
    (1+\cos\t\langle \nu,e\rangle) F^{-2}\ddot{F}^{kl,pq}\nabla_i h_{kl}\nabla^i h_{pq}\leq 0.
	\end{align*}
Since $\nabla_\mu H\leq 0$ on the boundary $\partial\Sigma_t$, arguing as Lemma \ref{s3:tech-lemma}, we know that $\nabla H=0$ and $\nabla^2H\leq 0$ at the spatial maximum point of $H$, no matter the maximum point is an interior point or a boundary point. Then, by the evolution equation \eqref{s3:evol-mean-curvature} of $H$, we have
	\begin{align}\label{s6.H1}
	\frac{d}{dt}\max_{M}H(p,t) \leq &\(\frac{1+\cos\t \langle \nu,e\rangle}{F^2}\dot{F}^{kl}(h^2)_{kl}+1\)H \nonumber\\
	&-\frac{2(1+\cos\t\langle \nu,e\rangle)}{F^3}|\nabla F|^2+\frac{2\cos\t}{F^2}\langle \nabla F, \nabla\langle \nu,e\rangle\rangle \nonumber\\
	&-(2+\cos\t \langle \nu,e\rangle) F^{-1}|A|^2.
	\end{align}
	Using \eqref{s6.F}, the first term on the right hand side of \eqref{s6.H1} is bounded by $C_1H$ for some positive constant $C_1$. The Weingarten equation and Cauchy-Schwarz inequality imply that
\begin{equation*}
  \langle \nabla F, \nabla\langle \nu,e\rangle\rangle \leq |\nabla F||A||e^T|,
\end{equation*}
where $e^T$ denotes the tangential part of $e$. By completing the square, the second line of \eqref{s6.H1} satisfies
	\begin{align*}
	 &-\frac{2(1+\cos\t\langle \nu,e\rangle)}{F^3}|\nabla F|^2+\frac{2\cos\t}{F^2}\langle \nabla F, \nabla\langle \nu,e\rangle\rangle \\
	\leq &-\frac{2(1+\cos\t\langle \nu,e\rangle)}{F^3}|\nabla F|^2+\frac{2\cos\t}{F^2}|\nabla F||A||e^T|\\
=&-\frac{2(1+\cos\t\langle \nu,e\rangle)}{F^3}\left(|\nabla F|-\frac{F\cos\theta |A||e^T|}{2(1+\cos\t\langle \nu,e\rangle)}\right)^2+\frac{\cos^2\theta}{2F(1+\cos\t\langle \nu,e\rangle)}|A|^2|e^T|^2\\
	\leq &\frac{\cos^2\theta}{2F(1+\cos\t\langle \nu,e\rangle)}|A|^2|e^T|^2.
	\end{align*}
It follows that
\begin{align*}
  \frac{d}{dt}\max_{M}H(p,t) \leq & C_1H -(2+\cos\t \langle \nu,e\rangle) F^{-1}|A|^2+\frac{\cos^2\theta}{2F(1+\cos\t\langle \nu,e\rangle)}|A|^2|e^T|^2\\
  = & C_1H-\frac{\sin^2\theta+3(1+\cos\theta\langle \nu,e\rangle)^2}{2F(1+\cos\t\langle \nu,e\rangle)}|A|^2\\
  \leq &C_1H-C_2|A|^2,
\end{align*}
where $C_2$ depends on the uniform bound of $F$ and $\theta$.  The maximum principle implies that $H$ is bounded from above by a uniform constant.
\end{proof}

By Theorem \ref{s4:thm-preserving-strict-convex} and Proposition \ref{s4:upper-bound-H}, we obtain the uniform two-sided bounds on the principal curvatures.
\begin{cor}\label{s4:principal-curv-upper-bound}
	Let $\Sigma_t$, $t\in [0,T)$ be the solution of the flow \eqref{s1:BGL-flow} starting from a strictly convex hypersurface $\Sigma_0$. Then there exists constants $C_1, C_2>0$ depending only on $\Sigma_0$ such that the principal curvatures of $\Sigma_t$ satisfy
	\begin{align}\label{s6.kpbd}
	0<C_1\leq \k_i(p,t)\leq C_2,\qquad i=1,\cdots,n
	\end{align}
for all $(p,t)\in M\times [0,T)$.
\end{cor}

\subsection{Convergence to a spherical cap}

At the beginning of \S \ref{sec:5}, we already proved the $C^0$ estimate of $\Sigma_t$ and deduced a positive lower bound on the support function $u=\langle x,\nu\rangle$. Then the $C^1$-estimate on the radial function $\varphi$ of $\Sigma_t$ follows from $\langle x,\nu\rangle=\frac{e^\varphi}{v}$ and
$$
|\bar{\nabla}\varphi|\leq v\leq C.
$$

Combining the $C^0$, $C^1$-estimates of $\varphi$ with the curvature estimates \eqref{s6.kpbd}, we show that $\varphi\in C^2(\mathbb S_+^n \times [0,T))$ and the scalar equation \eqref{s4:scalar-pde} is uniformly parabolic. Since $0\leq \cos\t<1$, the boundary condition \eqref{s6.dr2} satisfies the uniform oblique property. Then by the parabolic theory in \cite[Theorem 5]{Ural1991} (see also \cite[\S 14]{Lieb96}), we conclude the uniform $C^\infty$-estimates and the long-time existence of the solution to \eqref{s4:scalar-pde}. The convergence to a uniquely determined spherical cap with $\t$-capillary boundary as $t\ra \infty$ follows from the same argument as \cite[Proposition 4.12]{Wang-Weng-Xia2022}. This completes the proof of Theorem \ref{s1:thm-convergence}.

\subsection{Proof of  Theorem \ref{s1:thm-AF-ineq}}\label{sec:7}
     If the hypersurface $\Sigma$ is strictly convex, we run the flow \eqref{s1:BGL-flow} starting from $\Sigma$. For $1\leq \ell<k\leq n$, by Theorem \ref{s1:thm-convergence}, we know that the solution $\Sigma_t$ of the \eqref{s1:BGL-flow} converges smoothly to a spherical cap $C_{r_\infty,\t}$. By \eqref{s1:evol-V-ell} and \eqref{s1:evol-V-k}, we have
      \begin{align*}
      \mathcal{V}_{k,\t}(\widehat{\Sigma})=\mathcal{V}_{k,\t}(\widehat{C_{r_\infty,\t}}), \quad
      \mathcal{V}_{\ell,\t}(\widehat{\Sigma})\leq \mathcal{V}_{\ell,\t}(\widehat{C_{r_\infty,\t}}).
      \end{align*}
      Recall that
      $$
      \mathcal{V}_{m,\t}(\widehat{C_{r,\t}})=|\mathbb B_\t^{n+1}|r^{n+1-m}, \quad 0\leq m\leq n,
	  $$
	  we obtain
	  \begin{align}
	  \(\frac{\mathcal{V}_{k,\t}(\widehat{\Sigma})}{|\mathbb B_\t^{n+1}|}\)^\frac{1}{n+1-k}= & \(\frac{\mathcal{V}_{k,\t}(\widehat{C_{r_\infty,\t}})}{|\mathbb B_\t^{n+1}|}\)^\frac{1}{n+1-k}=r_\infty\nonumber\\
	  =&\(\frac{\mathcal{V}_{\ell,\t}(\widehat{C_{r_\infty,\t}})}{|\mathbb B_\t^{n+1}|}\)^\frac{1}{n+1-\ell}\nonumber\\
\geq &\(\frac{\mathcal{V}_{\ell,\t}(\widehat{\Sigma})}{|\mathbb B_\t^{n+1}|}\)^\frac{1}{n+1-\ell}.\label{s5.af}
	  \end{align}
	  Equality holds if and only if $H_kH_{\ell-1}=H_{k-1}H_{\ell}$ in \eqref{s1:evol-V-k}, which is equivalent to that all $\Sigma_t$ are spherical caps and in particular the initial hypersurface $\Sigma$ is a spherical cap. When $\Sigma$ is convex but not strictly convex, we can approximate $\Sigma$ by a family of strictly convex hypersurfaces $\Sigma_\varepsilon$ as $\varepsilon\to 0$. Then the inequality \eqref{s1:AF-inequality-capillary} follows by approximation. The equality characterization can be proved by using an argument of \cite{GL09}.

Finally, the inequality \eqref{s1:AF-inequality-capillary} for $0=\ell<k\leq n$ follows from combining \eqref{s1.iso} and \eqref{s5.af} for $1=\ell<k\leq n$.  This completes the proof of Theorem \ref{s1:thm-AF-ineq}.

\appendix
\section{Approximation result}
In this appendix, we prove that a convex hypersurface $\Sigma$ with $\theta$-capillary boundary in the half-space can be approximated by a sequence of strictly convex hypersurfaces with $\theta$-capillary boundary in the half-space in the $C^{2,\alpha}$ sense. Note that the free boundary case (i.e. $\theta=\frac{\pi}{2}$) in the unit ball was treated by Lambert and Scheuer in \cite{Lambert-Scheuer2017}. We adapt their idea and include a proof for the case of capillary hypersurfaces in the half-space.
	
Let $\Sigma\subset \-{\mathbb R}_{+}^{n+1}$ be a convex hypersurface with capillary boundary supported on $\partial\-{\mathbb R}_{+}^{n+1}$ at a contact angle $\theta\in (0,\frac{\pi}{2}]$, which is given by the embedding $x_0:M \ra \-{\mathbb R}_{+}^{n+1}$ of a compact manifold $M$ with non-empty boundary. We consider the mean curvature flow
	\begin{equation}\label{flow-mean}
		\left\{\begin{aligned}
			\partial_t x&=-H\nu+V, \qquad \text{in}\quad M \times[0,T),\\
			\langle\bar{N}\circ x,\nu\rangle&=\cos(\pi-\theta), \qquad \text{on}\quad \partial M \times[0,T),\\
			x(\cdot,0)&=x_0(\cdot) \qquad \text{on} \quad M,
		\end{aligned}\right.
	\end{equation}
and denote  $\Sigma_t=x(M,t)$, where $V$ is tangential component of $\partial_t x$ and satisfies $V\vert_{\partial\Sigma_t}=-H\cot\theta\mu$.

The short time existence of flow \eqref{flow-mean} can be obtained by a similar argument as in \cite{Stahl1996-2} by Stahl.
	\begin{thm}\label{short time}
		For any $\alpha\in(0,1)$, the flow \eqref{flow-mean} admits a unique solution:
		\begin{equation*}
			x(\cdot,t)\in C^{\infty}(M\times(0,\delta])\cap C^{2+\alpha,1+\frac{\alpha}{2}}(M\times[0,\delta]),
		\end{equation*}
		where $\delta>0$ is a small constant.
	\end{thm}
	
	Then we can prove the following approximation result:
	\begin{thm}
		Suppose that $\Sigma_t$, $t\in [0,\delta)$ is a solution to the flow \eqref{flow-mean} starting from the convex hypersurface $\Sigma$. Then $\Sigma_t$ is strictly convex for all time $t>0$ as long as the flow exists.
	\end{thm}
	\begin{proof}
		Note that the height function $\langle x,e_{n+1}\rangle$ attains its global maximum in the interior point of $\Sigma$, we can find a strictly convex point in the interior by attaching a large sphere to $\Sigma$ from above.
		
		Let
		\begin{equation*}
			\chi(x,t)=\min_{|\xi|=1}{h_{ij}\xi^i\xi^j}.
		\end{equation*}
		Since $h_{ij}$ is smooth, the function $\chi(x,t)$ is Lipschitz continuous in space and by a cut-off function argument, we find a smooth function $\phi_0:M\to\mathbb{R}$ such that $0\leq\phi_0\leq\chi(x,0)$ and there exists an interior point $y$ such that $\phi_0(y)>0$. We extend the function $\phi_0$ to $\phi:M\times[0,\delta')\to\mathbb{R}$ by solving a linear parabolic PDE:
		\begin{equation}\label{flow-heat}
			\left\{\begin{aligned}
				\frac{\partial}{\partial t}\phi&=\Delta\phi+\nabla_V\phi, \qquad \text{in}\quad M \times[0,\delta'),\\
				\nabla_{\mu}\phi&=0 \qquad \text{on}\quad \partial M \times[0,\delta'),\\
				\phi(\cdot,0)&={\phi}_0(\cdot) \qquad \text{on} \quad M,
			\end{aligned}\right.
		\end{equation}
		where $\Delta$ and $\nabla$ are Laplacian operator and Levi-Civita connection with respect to the induced  metric on $\Sigma_t=x(M,t)$ of the flow \eqref{flow-mean}. The solution $\phi$ of \eqref{flow-heat} exists at least for a short time interval $[0,\delta')$. By the strong maximum principle for scalar functions (see \cite[Corollary 3.2]{Stahl1996-2}), we have $\phi>0$ for all $x\in M$ and $t\in(0,\delta')$.

We take $\tau=\frac{1}{2}\min\{\delta,\delta'\}$ and consider the tensor:
		\begin{equation}
			M_{ij}=h_{ij}-\phi g_{ij}
		\end{equation}
		for the time interval $t\in[0,\tau)$. By the construction of $\phi_0$, we see that $M_{ij}\geq 0$ at time $t=0$. We now apply the tensor maximum principle (i.e. Theorem \ref{s1:thm-max principle}) to deduce that $M_{ij}\geq 0$ is preserved along the flow \eqref{flow-mean}.
		
		By a direct computation using Proposition \ref{s3:prop-evol-general-flow} for $\mathcal{F}=-H$, we have:
		\begin{equation}
			\frac{\partial}{\partial t} M_{ij}=\Delta M_{ij}+\nabla_{V} M_{ij}+N_{ij},
		\end{equation}
		where
		\begin{align*}
			N_{ij}=&|A|^2 M_{ij}+|A|^2\phi g_{ij}-2HM_i^k M_{kj}\nonumber\\
&-2H\phi M_{ij}+M_j^k\nabla_{i}{V_k}+M_i^k\nabla_{j}{V_k}.
		\end{align*}
We easily see that whenever $M_{ij}\geq0$ and $M_{ij}\xi^i=0$ at a point, we have
		\begin{equation}
			N_{ij}\xi^i\xi^j=|A|^2\phi\geq 0,
		\end{equation}
and thus the condition \eqref{conditon-MP} is satisfied. The boundary condition \eqref{eq-bou} can be checked similarly as in Theorem \ref{s4:thm-preserving-strict-convex}, using Proposition \ref{s2:prop-second-fundamental-form} and the fact that
\begin{equation*}
  \nabla\mu H=\cot\theta h_{\mu\mu}H
\end{equation*}
on the boundary $\partial\Sigma_t$ (see \eqref{s3:DF}). Hence Theorem \ref{s1:thm-max principle} implies that $M_{ij}\geq 0$ is preserved along the flow \eqref{flow-mean} for time $t\in[0,\tau)$, and it follows that $\Sigma_t$ is strictly convex for time interval $(0,\tau)$.
		
		To show the strict convexity for the whole time interval $(0,\delta)$, we fix a time $t_0\in(0,\tau)$. Since $\Sigma_t$ is strictly convex at time $t_0$, then there exists a constant $\varepsilon>0$, such that $h_{ij}\geq \varepsilon g_{ij}$ holds everywhere on $\Sigma_{t_0}$. A similar procedure as above can be used to show that $\tilde{M}_{ij}=h_{ij}-\varepsilon g_{ij}\geq 0$ is preserved along the flow \eqref{flow-mean} for all time $t>t_0$ as long as the flow \eqref{flow-mean} exists, which finishes the proof.
	\end{proof}

\noindent\textbf{Data Availability Statement.} Data sharing not applicable to this article as no datasets were generated or
analysed during the current study.

\noindent\textbf{Conflict of interest.} On behalf of all authors, the corresponding author states that there is no conflict of
interest.


\begin{thebibliography}{}


\bib{And2007}{article}{
	author={Andrews, Ben},
	title={Pinching estimates and motion of hypersurfaces by curvature functions},
	journal={J. reine angew. Math.},
	volume={608},
	year={2007},
	pages={17--33},
}

\bib{ACW20}{article}{
	author={Andrews, Ben},
    author={Chen, Xuzhong},
	author={Wei, Yong},
	title={Volume preserving flow and Alexandrov-Fenchel type inequalities in hyperbolic space},
	journal={ J. Eur. Math. Soc. (JEMS)},
	volume={23},
	year={2021},
    number={7},
	pages={2467--2509},
}


\bib{AW18}{article}{
	author={Andrews, Ben},
	author={Wei, Yong},
	title={Quermassintegral preserving curvature flow in hyperbolic space},
	journal={ Geom. Funct. Anal. },
	volume={28},
	year={2018},
    number={5},
	pages={1183--1208},
}

\bib{AHL2020}{article}{
	author={Andrews, Ben},
	author={Hu, Yingxiang},
	author={Li, Haizhong},
	title={Harmonic mean curvature flow and geometric inequalities},
	journal={Adv. Math.},
	volume={375},
	year={2020},
	pages={Paper No. 107393},
}

\bib{Andrews-McCoy-Zheng}{article}{
   author={Andrews, Ben},
   author={McCoy, James},
   author={Zheng, Yu},
   title={Contracting convex hypersurfaces by curvature},
   journal={Calc. Var. Partial Differential Equations},
   volume={47},
   date={2013},
   number={3-4},
   pages={611--665},
}


\bib{BGL}{article}{
	author={Brendle, Simon},
	author={Guan, Pengfei},
	author={Li, Junfang},
	title={An inverse curvature type hypersurface flow in space forms}
	note={preprint},
}

\bib{CS22}{article}{
	author={Chen, Min},
	author={ Sun, Jun},
	title={Alexandrov-Fenchel type inequalities in the sphere},
	journal={Adv. Math.},
	volume={397},
	year={2022},
	pages={Paper No. 108203},
}

\bib{ChenGLS22}{article}{
   author={Chen, Chuanqiang},
   author={Guan, Pengfei},
   author={ Li, Junfang},
   author={Scheuer, Julian},
   title={A fully-nonlinear flow and quermassintegral inequalities in the sphere},
   journal={ Pure Appl. Math. Q. },
   volume={18},
   date={2022},
   number={2},
   pages={437--461},
}

\bib{CM2011}{book}{
  author={Colding, Tobias Holck},
  author={Minicozzi, William P., II},
  title={A course in minimal surfaces },
  series={Graduate Studies in Mathematics},
  volume={121},
  year={2011},
  publisher={American Mathematical Society, Providence, RI}
}




\bib{GWW14}{article}{
	author={Ge, Yuxin},
	author={Wang, Guofang},
	author={Wu, Jie},
	title={Hyperbolic Alexandrov-Fenchel quermassintegral inequalities II},
	journal={ J. Differential Geom. },
	volume={98},
	year={2014},
    number={2},
	pages={237--260},
}

\bib{Ger1990}{article}{
    author={Gerhardt, Claus},
    title={Flow of nonconvex hypersurfaces into spheres},
    journal={ J. Differential Geom. },
    volume={32},
    year={1990},
    number={1},
    pages={ 299--314},
}

\bib{Ger06}{book}{
    author={Gerhardt, Claus},
    title={Curvature problems},
    series={ Series in Geometry and Topology},
    volume={39},
    year={2006},
    publisher={International Press, Somerville, MA},
}

\bib{guan2014}{book}{
  author={Guan, Pengfei},
  title={Curvature measures, isoperimetric type inequalities and fully nonlinear PDEs },
  series={Fully Nonlinear PDEs in Real and Complex Geometry and Optics, Lecture Notes in Mathematics},
  pages={47--94},
  year={2014},
  publisher={Springer}
}

\bib{GL09}{article}{
	author={Guan, Pengfei},
	author={Li, Junfang},
	title={The quermassintegral inequalities for $k$-convex starshaped domains},
	journal={Adv. Math.},
	volume = {221},
	pages = {1725--1732},
	year = {2009},
}

\bib{GL-2021}{article}{
	author={Guan, Pengfei},
	author={Li, Junfang},
	title={Isoperimetric type inequalities and hypersurface flows},
	journal={J. Math. Study},
	year={2021},
	volume={54},
	number={1},
	pages={56--80},
}

\bib{Hamilton1982}{article}{
	author={Hamilton, Richard},
	title={Three-manifolds with positive Ricci curvature},
	journal={J. Differential Geom.},
	volume={17},
	year={1982},
	pages={255--306},
}

\bib{Hamilton}{article}{
	author={Hamilton, Richard},
	title={Four-manifolds with positive curvature operator},
	journal={J. Differential Geom.},
	volume={24},
	year={1986},
	pages={153-–179},
}

\bib{HL18}{article}{
	author={Hu, Yingxiang},
	author={Li, Haizhong},
	title={Geometric inequalities for hypersurfaces with nonnegative sectional curvature in hyperbolic space},
	journal={Calc. Var. Partial Differential Equations},
	volume={58},
	year={2019},
    number={2},
	pages={paper no. 55},
}

\bib{HLW2020}{article}{
	author={Hu, Yingxiang},
	author={Li, Haizhong},
	author={Wei, Yong},
	title={Locally constrained curvature flows and geometric inequalities in hyperbolic space},
	journal={Math. Ann.},
	volume={382},
	year={2022},
	pages={1425-–1474},
}

\bib{HP99}{book}{
author={Huisken, Gerhard},
author={Polden, Alexander},
title={Geometric evolution equations for hypersurfaces. Calculus of variations and geometric evolution problems},
pages={45--84},
series={Lecture Notes in Math.},
volume={1713},
year={1999},
}

\bib{Lambert-Scheuer2017}{article}{
			author={Lambert, Ben},
			author={Scheuer, Julian},
			title={A geometric inequality for convex free boundary hypersurfaces in the unit ball},
			journal={Proc. Amer. Math. Soc.},
			volume={145},
			year={2017},
			number={9},
			pages={4009--4020},
}


\bib{LWX14}{article}{
    author={Li, Haizhong},
    author={Wei, Yong},
    author={Xiong, Changwei},
    title={A geometric inequality on hypersurface in hyperbolic space},
    journal={ Adv. Math.},
    volume={253},
    year={2014},
    pages={152--162},
}


\bib{Lieb96}{book}{
author={Lieberman, Gary M.},
title={Second order parabolic differential equations},
publisher={World Scientific Publishing Co., Inc., River Edge, NJ},
year={1996},
}

\bib{Maggi}{book}{
  author={Maggi, Francesco},
  title={Sets of finite perimeter and geometric variational problems},
  subtitle={An introduction to geometric measure theory},
  series={Cambridge Studies in Advanced Mathematics},
  volume={135},
  year={2012},
  publisher={Cambridge University Press, Cambridge},
}

\bib{MS15}{article}{
    author={Makowski, Matthias},
    author={Scheuer, Julian},
    title={Rigidity results, inverse curvature flows and Alexandrov-Fenchel type inequalities in the sphere},
    journal={ Asian J. Math.},
    volume={20},
    year={2016},
    number={5},
    pages={869--892},
}

\bib{McC05}{article}{
    author={McCoy, James A.},
    title={Mixed volume preserving curvature flows},
    journal={Calc. Var. Partial Differential Equations},
    volume={24},
    year={2005},
    number={2},
    pages={131--154},
}


\bib{Scheuer-Wang-Xia2018}{article}{
    author={Scheuer, Julian},
    author={Wang, Guofang},
    author={Xia, Chao},
    title={Alexandrov-Fenchel inequalities for convex hypersurfaces with free boundary in a ball},
    journal={J. Differential Geom.},
    volume={120},
    year={2022},
    pages={345--373},
}

\bib{Schn}{book}{
  author={Schneider, Rolf},
  title={Convex bodies: the Brunn-Minkowski theory},
  edition={Second expanded edition},
  series={Encyclopedia of Mathematics and its Applications},
  volume={151},
  year={2014},
  publisher={Cambridge University Press, Cambridge},
}

\bib{Stahl1996-2}{article}{
	author={Stahl, Axel},
	title={Regularity estimates for solutions to the mean curvature flow with a Neumann boundary condition},
	journal={Calc. Var. Partial Differential Equations},
	volume={4},
	pages={385--407},
	year={1996},
}

\bib{Ural1991}{article}{
	author={Ural'tseva, Nina N.},
	title={A nonlinear problem with an oblique derivative for parabolic equations},
	journal={Zap. Nauchn. Sem. Leningrad. Otdel. Mat. Inst. Steklov. (LOMI) \textbf{188} (1991);  Kraev. Zadachi Mat. Fiz. i Smezh. Voprosy Teor. Funktsi\u{i}. 22, 143--158, 188; translation in
J. Math. Sci. \textbf{70} (1994), no. 3, 1817--1827 },
}


\bib{Urbas1991}{article}{
	author={Urbas, John I. E.},
	title={On the expansion of starshaped hypersurfaces by symmetric functions of their principal curvatures},
	journal={Math. Z.},
	volume={205},
	year={1990},
    number={3},
    pages={355--372},
}

\bib{Urbas1991b}{article}{
	author={Urbas, John I. E.},
	title={An expansion of convex hypersurfaces},
	journal={J. Differential Geom.},
	volume={33},
	year={1991},
    number={1},
    pages={91--125},
}


\bib{Wang-Weng-Xia2022}{article}{
	author={Wang, Guofang},
	author={Weng, Liangjun},
	author={Xia, Chao},
	title={Alexandrov-Fenchel inequalities for convex hypersurfaces in the half-space with capillary boundary},
    journal={Math. Ann. (online first)},
	eprint={https://doi.org/10.1007/s00208-023-02571-4},
	year={2023},
}

\bib{WWX23}{article}{
	author={Wang, Guofang},
	author={Weng, Liangjun},
	author={Xia, Chao},
	title={A Minkowski-type inequality for capillary hypersurfaces in a half-space},
	eprint={arXiv:2209.13516v3},
}

\bib{Wang-Xia2014}{article}{
    author={Wang, Guofang},
    author={Xia, Chao},
    title={Isoperimetric type problems and Alexandrov-Fenchel type inequalities in the hyperbolic space},
    journal={Adv. Math.},
    volume={259},
    year={2014},
    pages={532--556},
}


\bib{Wei19}{article}{
   author={Wei, Yong},
   title={New pinching estimates for inverse curvature flows in space forms},
   journal={J. Geom. Anal.},
   volume={29},
   date={2019},
   number={2},
   pages={1555--1570},
}

\bib{WeiX15}{article}{
    author={Wei, Yong},
    author={Xiong, Changwei},
    title={Inequalities of Alexandrov-Fenchel type for convex hypersurfaces in hyperbolic space and in the sphere},
    journal={ Pacific J. Math. },
    volume={277},
    pages={219--239},
    year={2015},
    number={1},
}

\bib{WengX21}{article}{
    author={Weng, Liangjun},
    author={Xia, Chao},
    title={The Alexandrov-Fenchel inequalities for convex hypersurfaces with capillary boundary in a ball},
    journal={Trans. Amer. Math. Soc.},
    volume={375},
    pages={8851--8883},
    year={2022},
    number={12},
}

\end{thebibliography}
\end{document}